\documentclass{amsart}
\makeatletter
\@namedef{subjclassname@2020}{%
  \textup{2020} Mathematics Subject Classification}
\makeatother
\usepackage{tikz-cd}
\usepackage{tikz}
\tikzset{every picture/.style={line width=0.11mm}}
\newcommand{\oPerpSymbol}{\begin{tikzpicture}[scale=0.134]
  \draw (0,-0.5)--(0,1); \draw (-0.866,-0.5)--(0.866,-0.5);
  \draw (0,0) circle [radius=1];
\end{tikzpicture}}

\newcommand{\oPerp}{\mathbin{\raisebox{-1pt}{\oPerpSymbol}}}

\usepackage{graphicx}
\usepackage{hyperref}

\usepackage{amssymb}
\usepackage{amsthm}
\usepackage{textcomp}
\usepackage{calrsfs}
\DeclareMathAlphabet{\pazocal}{OMS}{zplm}{m}{n}
\usepackage{enumitem}

\usepackage{marginnote}

\usepackage{scalerel}
\def\abs#1{\vert #1\vert}
\def\st{\mathop{\hbox{\rm st}}}
\def\shrinkage{-2.4mu}
\def\vecsign#1{\rule[1.388\LMex]{\dimexpr#1-2.5pt}{.36\LMpt}%
  \kern-6.0\LMpt\mathchar"017E}

\def\theraysign#1{\rule{0pt}{17\LMpt}\rule[1.384\LMex]{\dimexpr#1-2.5pt}{.40\LMpt}%
  \kern-6.0\LMpt\mathchar"017E}
\def\raysign#1{\rule{0pt}{7\LMpt}\smash{%
  \SavedStyle\mkern-\shrinkage\theraysign{#1}}}
\def\ray#1{\ThisStyle{\setbox0=\hbox{$\SavedStyle#1$}%
  \def\useanchorwidth{T}\stackon[-4.2\LMpt]{\SavedStyle#1}{\,\raysign{\wd0}}}}
\usepackage{stackengine,amsmath}
\stackMath
\usepackage{graphicx}

\def\neg{\mathop{\text{neg}}}
\def\operpf#1{\mathop{\oPerp^{\F}_{#1}}}
\def\sf{\bot^{\!\!\tiny\F}}
\def\roots{\mathop{\hbox{\bf roots}}}
\def\remove#1{}

\def\esc#1{\langle #1\rangle}
\def\escd#1{\langle\!\langle #1\rangle\!\rangle}
\def\esct#1{\langle\!\langle\!\langle #1\rangle\!\rangle\!\rangle}
\def\F{\mathcal F}
\newcommand{\K}{{\mathbb{K}}}
\newcommand{\FC}{{\mathbb{F}}}
\newcommand{\U}{{\mathfrak{U}}}

\def\N{\mathbb{N}}
\def\R{\mathbb{R}}
\def\C{\mathbb{C}}
\def\hipR{{^*\R}}
\def\hipC{{^*\C}}
\def\hipK{{^*\K}}

\def\rad{\mathop{\hbox{\bf rad}}}

\def\l{\lambda}
\def\m{\mu}
\def\a{\alpha}
\def\b{\beta}

\def\ch{\mathop{\hbox{\rm char}}}
\def\so{simultaneously orthogonalizable}
\def\sd{simultaneously diagonalizable}
\def\T{\mathcal T}

\def\span{\mathop{\hbox{\rm span}}}

\def\End{\mathop{\hbox{\rm End}}}

\def\PF{\hbox{\bf P}_{\mathbf{F}}}
\def\o{\omega}
\def\r{\mathfrak r}
\def\ii{{\boldsymbol{i}}}
\def\oo{\otimes}
\def\path{\mathop{\hbox{\rm path}}}
\def\escemp{\esc{\cdot,\cdot}}

\newtheorem{lemma}{Lemma}
\newtheorem{proposition}{Proposition}
\newtheorem{corollary}{Corollary}
\newtheorem{definition}{Definition}

\newtheorem{theorem}{Theorem}
\newtheorem{example}{Example}

\newtheorem{remarkx}{Remark}
\newenvironment{remark}
  {\pushQED{\qed}\remarkx}
  {\popQED\endremarkx}

\title[Simultaneous orthogonalization in infinite-dimensional spaces]{Simultaneous orthogonalization of inner products in infinite-dimensional vector spaces}
\author[Y. Cabrera]{Yolanda Cabrera Casado}

\author[C. Gil]{Crist\'obal Gil Canto}

\author[D. Mart\'{\i}n]{Dolores Mart\'in Barquero}

\author[C. Mart\'{\i}n]{C\'andido Mart\'in Gonz\'alez}

\address{Departamento de Matem\'atica Aplicada, E.T.S. Ingenier\'\i a Inform\'atica, Universidad de M\'alaga, Campus de Teatinos s/n. 29071 M\'alaga.   Spain. }
\email{yolandacc@uma.es}

\address{Departamento de Matem\'atica Aplicada, E.T.S. Ingenier\'\i a Inform\'atica, Universidad de M\'alaga, Campus de Teatinos s/n. 29071 M\'alaga.   Spain.}
\email{cgilc@uma.es}

\address{Departamento de Matem\'atica Aplicada, Escuela de Ingenier\'\i as Industriales, Universidad de M\'alaga, Campus de Teatinos s/n. 29071 M\'alaga.   Spain.}
\email{dmartin@uma.es}

\address{ Departamento de \'Algebra Geometr\'{\i}a y Topolog\'{\i}a, Fa\-cultad de Ciencias, Universidad de M\'alaga, Campus de Teatinos s/n. 29071 M\'alaga.   Spain.}
\email{candido\_m@uma.es}

\thanks{The  authors are supported by the Junta de Andaluc\'{i}a  through projects  FQM-336 and UMA18-FEDERJA-119 and  by the Spanish Ministerio de Ciencia e Innovaci\'on   through project  PID2019-104236GB-I00,  all of them with FEDER funds.}

\begin{document}


\maketitle

\begin{abstract}
For an arbitrary field $\mathbb{K}$ and a family of inner products in a $\mathbb{K}$-vector space $V$ of arbitrary dimension, we study necessary and sufficient conditions in order to have an orthogonal basis relative to all the inner products. If the family contains a nondegenerate element plus a compatibility condition, then under mild hypotheses the simultaneous orthogonalization can be achieved. So we investigate several constructions whose purpose is to add a nondegenerate element to a degenerate family and we study under what conditions the enlarged family is nondegenerate.
\end{abstract}

\section{Introduction and preliminaries}
Historically the research about the problem of simultaneous orthogonalization of two inner products over a $\K$-vector space of finite dimension, with $\K$ a field satisfying different restrictions, has been largely studied by several authors, see, for example, the works \cite{Finsler}, \cite{Calabi}, \cite{Wo}, \cite{Uhligart}, \cite{Greub} and \cite{Becker}. The analogous problem for a family with two or more inner products over a $\K$-vector space of finite dimension has been considered in \cite{BMV} when the ground field is the real or complex numbers. In general, for an arbitrary field $\K$ and a finite-dimensional vector space over $\K$, a thorough study has been recently realised in \cite{CGMM}. In fact, \cite{CGMM} has motivated a natural follow-up to the simultaneous orthogonalization of families of inner products over a vector space of arbitrary dimension over an arbitrary field $\K$.
\newline
\indent The paper is organized as follows: Section 2 contains the connection between the notion of roots and simultaneous diagonalization of families of endomorphisms. This is preparatory for the results on simultaneous orthogonalization that we pursue. Section 3 deals firstly with some topological notions that we will need in order to introduce the concept of a nondegenerate family of inner products. 
This latest notion will lead us to divide the problem of simultaneous orthogonalization into two procedures, depending on whether the family of inner products is nondegenerate or not. 
The nondegenerate case is contained in Theorem \ref{ogacem} and its corollaries. The degenerate case is studied in the remaining subsection (see Theorem \ref{oepem} and its corollary). In Section 4 we consider different constructions whose goal is to modify a given family $\F$ of inner products, where possibly all of them are degenerated, as to get a new family containing a nondegenerate inner product and whose orthogonalization is
induced by an orthogonalization of $\F$. To this end, we use several philosophies, some of them exploit the idea of adding a suitable linear combination which turns out to be nondegenerate. Others use an ultrafilter construction that we particularize at the end of the paper to the case in which the ground field is $\R$ or $\C$. To do so, we introduce new concepts related to the idea of how elements on the vector space behave, such as pathological (or negligible in the real-complex case) elements. This will guarantee, together with the simultaneous orthogonalization, the nondegenerancy of new certain families of inner products constructed from the original one. We are talking about Theorem \ref{pajaro} and Corollary \ref{playita} when the base field is arbitrary and Theorem \ref{wwe} for $\R$ or $\C$.

From now on we will use the following basic definitions and notations.  The complementary of a subset $A$ of $X$ will be denoted by $X\setminus A$ or by $\complement A$ (if there is no possible ambiguity). We write the cardinal of $X$ as $\rm{card}(X)$. If $I$ is a set, $\PF(I)$ will denote the set of finite parts of $I$. \newline
\indent We consider the naturals with zero included: $\N=\{0,1,\ldots\}$ and use the notation $\N^*=\N\setminus\{0\}$. For a field $\K$ we use any of the notations $\K^*$ or $\K^\times$ for the multiplicative group $\K\setminus\{0\}$. \newline 
\indent All through this work, we denote by $V$  a $\K$ vector space where $\K$ is the field. We use ${\rm char}(\K)$ for the characteristic of $\K$. If $V$ is a $\K$-vector space and $S$ is a subset of $V$, we  write $\span(S)$ to denote the subspace generated by $S$. A vector subspace $H$ of a vector space $V$ is a {\it hyperplane} if  $\dim(V/H)=1$. A vector subspace $H$ of a vector space $V$ is {\it proper} if it is other than the whole space and $H$ is {\it trivial} if it is the null space.\newline
\indent By an {\it inner product} on a $\K$-vector space $V$ we will mean a symmetric bilinear form $\esc{\cdot,\cdot}\colon V\times V\to \K$. And with $(V,\esc{\cdot,\cdot})$ we refer to the so-called  {\it inner product $\K$-vector space}. We denote by $V^{\circ}:=\{v \in V \colon  \esc{v,v}=0\}$ the \emph{set of isotropic} vectors of $V$. Observe that if ${\rm char}(\K)=2$, then $V^\circ$ is a subspace of $V$. If $U$ is a vector subspace of $V$, we can consider the inner product $\K$-vector space $(U,\esc{\cdot,\cdot}\vert_U)$ where $\esc{\cdot,\cdot}\vert_U$ is the restriction of $\esc{\cdot,\cdot}$ to $U$.  We denote by $U^{\perp}:=\{x \in V \colon \esc{x,U}=0\}$. We say that $U$ is  {\it $\perp$-closed} if $U^{\perp \perp}=U$.  For an element $u\in V$ we set $u^\bot:=(\K u)^\bot$.
Define $\rad(V,\esc{\cdot,\cdot}):=\{x\in V\colon \esc{x,V}=0\}$. If there is no possible confusion we will write $\rad(\esc{\cdot,\cdot})$. When $\rad(\esc{\cdot,\cdot})=0$ we say that $(V,\esc{\cdot,\cdot})$ is \emph{nondegenerate}. In this case we may consider the {\it dual pair} $(V,V)$ in the sense of \cite[\S IV, section 6, Definition 1, p. 69]{Jacobson}. For us, $\oPerp ^{\esc{\cdot,\cdot}}$ means orthogonal direct sum relative to the inner product $\esc{\cdot,\cdot}$. If it is clear we write $\oPerp$ instead of  $\oPerp ^{\esc{\cdot,\cdot}}$. A family of inner products $\F=\{\esc{\cdot,\cdot}_i\}_{i\in I}$ is said to be \emph{simultaneously orthogonalizable} if there exists a basis $\{v_j\}_{j\in\Lambda}$ of $V$ such that for any $i\in I$ we have   
$\esc{v_j,v_k}_i=0$ whenever $j\ne k$. 
Finally, we denote by $\oPerp^{\F}$ the orthogonal direct sum relative to any inner product of $\F$ and  by $\sf $ the perpendicularity relationship for any $\esc{\cdot,\cdot}\in\F$.

\section{Roots and root spaces}\label{gatito}
Let $\T\subset \End(V)$ where $\End(V)$ is the set of $\K$-linear maps $V\to V$. A map $\a\colon\T\to\K$ is said to be a {\em root} for $\T$ if and only if there is some nonzero $v\in V$ such that for any $T\in\T$ one has $T(v)=\a(T)v$. Any such vector $v$ is called a {\em root vector} of $\T$ (relative to $\a$) and the set of all $w\in V$ such that $T(w)=\a(T)w$ (for any $T\in\T$) is called the {\em root space} of $\a$ and denoted $V_\a$. Observe that the constant map $0\colon\T\to\K$ such that $T\mapsto 0$ is a root of $\T$ if and only if 
$\cap_{T\in\T}\ker(T)\ne 0$. Now let $\K\T$ denote the $\K$-subspace of $\End(V)$ generated by $\T$. 
We will denote by $\roots(\T)$ the set of all roots of $\T$. Then we have: 
\begin{lemma} 
There is a canonical bijection $\tau\colon \roots(\T)\cong\roots(\K\T)$
such that $\tau(\a)\vert_\T=\a$ for any $\a\in\T$.
\end{lemma}
\begin{proof}
Given a root $\a\in\roots(\T)$ we must extend it to a root $\tau(\a)$ of $\K\T$. To do this, we can fix a basis of $\K\T$ given by $\{T_i\}_{i\in I}\subset\T$ and then, for a generic element $\sum_i \l_i T_i\in\K\T$ with $\lambda_i\in\K$ and $T_i\in\T$ we define $\tau(\a)(\sum_i\l_i T_i):=\sum_i\l_i\a(T_i)$. This definition does not depend on the chosen basis. Indeed, if we consider a different basis $\{S_j\}\subset \T$, then we must prove that when $\sum_i\l_i T_i=\sum_j\mu_j S_j$ (with $\l_i,\mu_j\in\K$) we get 
$\sum_i\l_i \a(T_i)=\sum_j\mu_j\a(S_j)$. To check this equality, we take a root vector $v$ relative to $\a$ and apply both members of the equality $\sum_i\l_i T_i=\sum_j\mu_j S_j$ to $v$. Now the well defined map $\tau(\a)\colon\K\T\to\K$ is a root of $\K\T$ because any root vector of $\T$ relative to $\a$ is a root vector of $\K\T$ relative to $\tau(\a)$. So we have a map $\tau\colon\roots(\T)\to\roots(\K\T)$.
Also, note that $\tau(\a)(T_i)=\a(T_i)$ for any element $T_i$ of the basis $\{T_i\}$, hence the restriction of $\tau(\a)$ to $\T$ is $\a$. This proves that $\tau$ is injective. To see that $\tau$ is surjective consider a root $\b$ of $\K\T$ and let $\a:=\b\vert_\T$.
Then $\a$ is a root of $\T$ and
$\a(T_i)=\b(T_i)$ for any basic element $T_i$.
Since $\tau(\a)(T_i)=\a(T_i)$ also, we conclude that $\tau(\a)=\b$.
\end{proof}
So if $\T\subset\End(V)$, then essentially $\T$ and $\K\T$ have the same roots. This is why we can replace $\T$ with $\K\T$ in our work. Therefore, we can assume from the beginning that the set of linear maps $\T$ is in fact a vector subspace. 
In this setting, we see that any $\a\in\roots(\T)$ is a linear map: 
\begin{enumerate}[label=(\roman*)]
    \item If $T,S\in\T$ we consider a root vector $v$ of $\T$ relative to $\a$ and then $(T+S)(v)=\a(T+S)v$ and on the other hand $T(v)+S(v)=\a(T)v+\a(S)v$ so that 
    $\a(T+S)=\a(T)+\a(S)$.
    \item If $T\in\T$ we consider a root vector $v$ of $\T$ relative to $\a$ and $\l\in\K$, then $(\l T)(v)=\a(\l T)v$ and $\l T(v)=\l\a(T)v$ whence $\l\a(T)=\a(\l T)$. 
\end{enumerate}
So we have that $\roots(\T)$ is a subspace of the dual space.
\begin{lemma} \label{igual} Let $\T$ be a vector subspace of $\End(V)$, for two roots $\a,  \b$ of $\T$,  we have that  $V_\a=V_\b$  if and only if  $\a=\b$. 
\end{lemma}
\begin{proof}
Take a  root vector $v$ of $\T$ relative to $\a$. Then for any
$T\in\T$ we have $T(v)=\a(T)v$ but $v\in V_\b$ hence $T(v)=\b(T)v$ whence $\a(T)=\b(T)$ and since $T$ is arbitrary $\a=\b$. The other implication is straightforward.
\end{proof}

Recall that a family $\T \subset \End{(V)}$ is {\it simultaneously diagonalizable} if there exists a basis $B$ of $V$ such that $T(v)\in \K v$ for any $v \in B$.
\begin{proposition}\label{palmera} Let V be a $\K$-vector space and let $\T$  be a vector subspace of $\End(V)$. Then we have the following assertions:
\begin{enumerate}[label=(\roman*)]
    \item If $\T$ is  simultaneously diagonalizable and $B=\{v_i\}_{i\in I}$ a basis of $V$ diagonalizing all the elements of $\T$, define
for each $i\in I$ a map $\a_i\colon\T\to\K$ such that 
$\a_i(T)$ is the eigenvalue of $v_i$ relative to $T\in \T$. 
Now, from the indexed family of roots $\{\a_i\colon i\in I\}$ we eliminate repetitions obtaining a set $\Phi$. Then
$$V=\bigoplus_{\a_i\in \Phi}V_{\a_i}$$
where $V_{\a_i}:=\{x\in V \colon \forall \,  T\in\T, T(x)=\a_i(T)x\}$.
\item Conversely, if $V$ is a direct sum of root spaces relative to some set of roots of $\T$, then $\T$ is simultaneously diagonalizable.
\item \label{coco} If all the elements of $\T$ are self-adjoint relative to an inner product $\esc{\cdot, \cdot}$ defined on $V$, then the root spaces are
pairwise orthogonal.
\end{enumerate}
\end{proposition}
\begin{proof}

By the definition of $\a_i$ we have $T(v_i)=\a_i(T)v_i$ hence
$v_i\in V_{\a_i}$. Let us prove that for any $i$ we have
$$V_{\a_i}=\span(\{v_j \in B \colon V_{\a_j}=V_{\a_i}\}).$$
Let $0\ne x\in V_{\a_i}$ and write it in terms of the basis $B$. Thus $x=\sum_l k_l v_l$ (for scalars $k_l\in\K$ not all null). 
Then, for an arbitrary $T\in\T$, we have $\sum_l\a_i(T)k_l v_l=\sum_l k_l T(v_l)$,  whence for the $l$'s such that $k_l\ne 0$ we conclude that $\a_i(T)=\a_l(T)$  since $v_l\in V_{\a_l}$. Consequently $V_{\a_i}$ is the linear span of the $v_l$'s such that $\a_l=\a_i$.
Thus each $V_{\a_i}$ has a basis consisting of elements of $B$. This implies, applying Lemma \ref{igual}, that the sum $\sum_{\a_i\in\Phi}V_{\a_i}$ is direct. To show that it coincides with the whole $V$ take into account that for any $v_k\in B$ the root $\a_k$ is in $\Phi$ hence $v_k\in V_{\a_k}$.  The second assertion of the proposition is straightforward considering a basis of each root space and performing the union of these. To prove item \ref{coco},
consider two roots $\a,\b$ with $\a\ne\b$. Then there is some $T\in\T$ such that $\a(T)\ne\b(T)$. So for any
$x\in V_\a$ and $y\in V_\b$ we have 
$\a(T)\esc{x,y}=\esc{T(x),y}=\esc{x,T(y)}=\b(T)\esc{x,y}$. Since $\a(T)\ne\b(T)$ we conclude $\esc{x,y}=0$.
\end{proof}
\begin{corollary}
If $\T$ is a subspace of $\End(V)$ which is simultaneously diagonalizable, denote by $\hat\T$ the subalgebra of $\End(V)$ generated by $\T$. Then any root $\a\colon\T\to\K$ can be extended to a homomorphism of $\K$-algebras $\hat\a\colon\hat\T\to\K$.
\end{corollary}
\begin{proof}
The natural extension of $\a$ should be 
$$\hat\a(\sum_{i_1,\ldots,i_n}k_{i_1,\ldots,i_n}T_{i_1}\cdots T_{i_n}):=\sum_{i_1,\ldots,i_n}k_{i_1,\ldots,i_n}\a(T_{i_1})\cdots \a(T_{i_n}).$$ It is easy to prove that $\hat\a$ is well-defined.
\end{proof}

Observe that for each root $\a$, the projection of $V$ onto  $V_\a$ is an idempotent of $\End(V)$ and the collection of all such projections is a system of orthogonal idempotents whose sum is $1_V$. Concretely, the idempotents in \cite[Theorem 4.10 (c)]{IMR} are these projections onto root subspaces. This establishes a link of the current section with \cite{IMR}.

\section{Simultaneous orthogonalization in infinite dimension} 

In the first subsection we will consider the most favorable case of simultaneous orthogonalization, roughly speaking, the case in which the family contains a
nondegenerate inner product. Then we advance to the case in which there is
one inner product whose radical is contained in the radical of the remaining
members of the family. Before approaching the simultaneous orthogonalization of a family of inner products we would like to observe that in a finite-dimensional space endowed with two
inner products, with the radical of  the first one contained in the radical of the second one, the  latter  can be written
in terms of the former. This fact will play a substantial roll also in
the infinite-dimensional case.

A linear algebra elementary result states that if $R,S\colon V\to V$ are linear maps with $\ker(R)\subset\ker(S)$, then there is a linear map $P\colon V\to V$ such that $S=PR$. We use this result to prove the following lemma:

\begin{lemma}
Let $V$ be a finite-dimensional vector space with $\esc{\cdot,\cdot}_0$ and $\esc{\cdot,\cdot}_1$ two inner products on
$V$ such that 
$\rad(\esc{\cdot,\cdot}_0)\subset \rad(\esc{\cdot,\cdot}_1)$. Then $\esc{x,y}_1=\esc{T(x),y}_0$ for some linear map $T\colon V\to V$ and any $x,y \in V$. 
\end{lemma}
\begin{proof} Take a basis $B=\{v_i\}_{i=1}^n$ of $V$. Consider the matrices $(\esc{v_i,v_j}_0)_{i,j}$ and $(\esc{v_i,v_j}_1)_{i,j}$. Then 
$$ x (\esc{v_i,v_j}_0)_{i,j}=0 \Rightarrow x (\esc{v_i,v_j}_1)_{i,j}=0$$
for any row vector $x\in\K^n$. So there is an $n\times n$ matrix $P$ such that 
$(\esc{v_i,v_j}_1)_{i,j}=(\esc{v_i,v_j}_0)_{i,j} P$. If we write $P=(p_{ij})$
we have $$\esc{v_i,v_j}_1=\sum_k \esc{v_i,v_k}_0 p_{kj}=
\esc{v_i,\sum_k p_{kj}v_k}_0=\esc{v_i,T(v_j)}_0$$ where $T\colon V\to V$ is such that
$T(v_j)=\sum_k p_{kj}v_k$. Summarizing $\esc{x,y}_1=\esc{x,T(y)}_0$ which is equivalent to say $\esc{x,y}_1=\esc{T(x),y}_0$. 
\end{proof}
\medskip

In order to extend as far as possible the previous finite-dimensional result we will need to implement some fundamental topological ideas.

\medskip

Let $V$ be a vector space over a field $\K$ and the inner product $\esc{\cdot,\cdot} \colon V \times V \to \K$ such that $(V,V)$ is a dual pair.
We can consider the 
 topology whose basis of neighborhoods of an $x\in V$ is the given by the sets $x+\cap_{i=1}^n v_i^\bot$ for some finite collection $v_1,\ldots,v_n$ (see \cite[\S IV, section 6, Definition 2, p. 70]{Jacobson}). We will call this the $\esc{\cdot,\cdot}$-{\it topology} of $V$. 
 One can see that the closed subsets are those subspaces $U\subset V$ such that $U^{\bot\bot}=U$ (see \cite[\S IV, section 6, Proposition 1, p. 71]{Jacobson}). This can be applied to the ground field endowed with the inner product $\esc{\lambda,\mu}_{\K}:=\lambda \mu$ with $\lambda, \mu \in \K$. Of course the topology induced in $\K$ is the discrete topology.

\begin{definition} \rm
Let $X$ and $Y$ be topological spaces and $B\colon X\times X\to Y$ a map. We will say that $B$ is {\em partially continuous} if both maps $B(x,\_),B(\_,x)\colon X\to Y$ are continuous for every $x \in X$. 
\end{definition}

\begin{remark}\label{copioso}\rm 
Let $(V,\esc{\cdot,\cdot})$ be a nondegenerate inner product $\K$-vector space. Recall the definition
of the {\it topological dual} $V^*$, that is, all linear maps $V\to \K$ which are continuous considering in $V$ the $\esc{\cdot,\cdot}$-topology and the discrete one in $\K$. By  \cite[\S IV, section 7, Lemma, p.72]{Jacobson} each $f\in V^*$ is of the form $f=\esc{y,\_}$ for some $y\in V$. Furthermore, observe that $\esc{\cdot,\cdot}$ is partially continuous relative to itself. Indeed, for $x \in V$ let $f_x \colon V \to \K$ given by $f_x(v):=\esc{x,v}$. We will prove that $f_x$ is continuous finding an adjoint (see \cite[\S IV, section 7, Theorem 1, p. 72]{Jacobson}). So, write $(f_x)^{\sharp} \colon \K \to V$ defined linearly as $(f_x)^{\sharp}(1)=x$. For $\lambda \in \K$, we have $\esc{f_x(v),\l}_{\K}=\lambda f_x(v)= \lambda\esc{x,v}=\esc{(f_x)^{\sharp}(\lambda),v}$.
\end{remark}

\begin{proposition}\label{atupoi}
Let $V$ be an arbitrary $\K$-vector space provided with two inner products $\esc{\cdot,\cdot}_i$ ($i=0,1$). Assume that 
$\esc{\cdot,\cdot}_0$ is nondegenerate and endow $V$ with the 
$\esc{\cdot,\cdot}_0$-topology. Then:
\begin{enumerate}[label=(\roman*)]
    \item \label{atupoi1}The inner product  $\esc{\cdot,\cdot}_1$ is partially continuous if and only if there is a continuous linear map $T\colon V\to V$ such that $\esc{x,y}_1=\esc{T(x),y}_0$ for any $x,y\in V$.
    \item \label{atupoi2} If both inner products are \so , then $\esc{\cdot,\cdot}_1$ is partially continuous.
\end{enumerate}
\end{proposition}
\begin{proof}
If $\esc{\cdot,\cdot}_1$ is partially continuous, then for any $x\in V$ the linear map $f_x:=\esc{x,\_}_1$ is continuous. Whence by Remark \ref{copioso},
there is an unique element $a_x\in V$ verifying $f_x=\esc{a_x,\_}_0$. The uniqueness follows from nondegeneracy of  $\esc{\cdot,\cdot}_0$.
Defining $T\colon V\to V$ such that $T(x)=a_x$, we have 
\begin{equation}\label{reason}
\esc{x,\_}_1=\esc{T(x),\_}_0
\end{equation}
for any $x \in V$. It can be seen that $T$ is linear. Next we prove the continuity of $T$ or equivalently that it has an adjoint relative to $\esc{\cdot,\cdot}_0$. In fact, $T$ is self-adjoint:
$$\esc{T(x),y}_0=\esc{x,y}_1=\esc{y,x}_1=f_y(x)=\esc{a_y,x}_0=\esc{x,T(y)}_0.$$
Conversely, assume that the inner product $\esc{\cdot,\cdot}_1$ is written as 
$\esc{x,y}_1=\esc{T(x),y}_0$ for some continuous linear map $T\colon V\to V$ and arbitrary $x,y\in V$. In order to prove that 
$\esc{\cdot,\cdot}_1$ is partially continuous we need to check that for any $x\in V$ the map $f_x\colon V\to\K$ defined as $f_x(v):=\esc{x,v}_1$ is continuous. But this is equivalent again to prove that $f_x$ has an adjoint. Define the linear map 
$S\colon \K\to V$ such that $1\mapsto T(x)$. We check that this is the adjoint map of $f_x$:
$$\esc{f_x(v),\l}_{\K}=\l f_x(v)=\l \esc{x,v}_1=\l \esc{v,T(x)}_0=\esc{v,S(\l)}_0$$ for every $\l \in \K$ and $v \in V$, which proves that $f_x$ is continuous.\newline 
For the second item take an orthogonal basis $B=\{v_i\}_{i\in I}$ for both inner products. Since $\esc{v_i,v_i}_0=\l_i\ne 0$ denoting  $\m_i:=\esc{v_i,v_i}_1$ we can define the linear map $T\colon V\to V$ such that $T(v_i)=\frac{\m_i}{\l_i}v_i$ for any $i$. So, denoting $\delta_{ij}$ the Kronecker's delta, we have $$\esc{T(v_i),v_j}_0=\frac{\m_i}{\l_i}\esc{v_i,v_j}_0=\delta_{ij}\frac{\m_i}{\l_i}\l_i=\esc{v_i,v_j}_1$$
which gives $\esc{T(x),y}_0=\esc{x,y}_1$ for any $x,y\in V$. Furthermore $T$ is easily seen to be self-adjoint. Thus applying item \ref{atupoi1} we prove the statement.
\end{proof}

\subsection{Simultaneous orthogonalization of a nondegenerate family of inner products}\label{gallo}

In this subsection we approach the issue of finding necessary and sufficient conditions to the existence of a basis which simultaneously orthogonalizes a nondegenerate family of inner products. We start by defining a such nondegenerate family.

\begin{definition}\label{fuerade} \rm
Let $\F$ be a family of inner products in a vector space $V$ over $\K$. We will say that $\F$ is {\it nondegenerate} if there is some element in $\F$  whose radical is $0$, say $\esc{\cdot,\cdot}_0$,  such that any 
$\esc{\cdot,\cdot}\in\F$ is partially continuous relative to the $\esc{\cdot,\cdot}_0$-topology of $V$. On the contrary, we will refer to a {\it degenerate} family. 
\end{definition}

If a family $\F$ is nondegenerate we will use
the notation $\F=\{\esc{\cdot,\cdot}_i\}_{i\in I \cup \{0\}}$ to indicate that the specific inner product $\esc{\cdot,\cdot}_0$ is nondegenerate and the remaining are partially continuous relative to the $\esc{\cdot,\cdot}_0$-topology of $V$.

\begin{remark}\rm
If $V$ is finite-dimensional the condition on
the continuity is redundant so that $\F$ is nondegenerate if and only
if there is some nondegenerate inner product within $\F$. 
Regardless of the dimension of $V$, if $\F$ is \so\ the condition that any inner product in $\F$ is partially continuous relative to the 
$\esc{\cdot,\cdot}_0$-topology of $V$ 
is redundant in virtue of Proposition \ref{atupoi} \ref{atupoi2}.
\end{remark}

\begin{remark}\label{fragel}\rm If $V$ is arbitrarily dimensional and $\F=\{\escemp_i\}_{i\in I}$ is \so\ we can construct a new
 family $\F':=\F\sqcup\{\escemp_0\}$ by adding the inner product defined in the following way: take a basis $\{v_j\}$ which is orthogonal for $\F$ and write 
 $\esc{v_i,v_j}_0=\delta_{ij}$.
 The basis of the $v_j$'s orthogonalizes also to $\F'$.
 Furthermore the family $\F'$ is nondegenerate . 
 Thus a necessary condition for the simultaneous orthogonalization of a given family $\F$ is that by adding
 at most one element, we get a new family $\F'$ which is
 nondegenerate  and 
 \so. \end{remark}
 
 So the natural starting point for this study should be
 the nondegenerate families. In this sense we have the following result.

\begin{theorem}

\label{ogacem}
Assume that $\F=\{\esc{\cdot,\!\cdot}_i\}_{i\in I\cup\{0\}}$ is a nondegenerate family of inner products on the vector space $V$ of arbitrary dimension over $\K$. Then:
\begin{enumerate}[label=(\roman*)]
\item \label{ogacem1} For each $i\in I$
there is a linear map $T_i\colon V\to V$ such that $\esc{x,y}_i=\esc{T_i(x),y}_0$ \ for any $x,y\in V$. Furthermore, each $T_i$ is a self-adjoint operator of $(V,\esc{\cdot , \cdot}_0)$.
\item \label{ogacem2}$\F$ is \so\  if and only if there exists an orthogonal basis $B$ of $(V,\esc{\cdot , \cdot}_0)$ such that each $T_i$ (as in item \ref{ogacem1}) is diagonalizable relative to $B$.
\item \label{short} In case ${\ch}(\K)\ne 2$ and $\dim(V)$ is either finite or infinite countable, the family $\F$ is simultaneously orthogonalizable if and only if $\{T_i\}_{i\in I}$ (as in item \ref{ogacem1}) is \sd.
\item \label{short2} In case ${\ch}(\K)= 2$ and $\dim(V)$ is either finite or infinite countable, if  $\F$ is \so , then $\{T_i\}_{i\in I}$ (as in item \ref{ogacem1}) is simultaneously diagonalizable. Conversely, when $\{T_i\}_{i\in I}$ is simultaneously diagonalizable
we can consider the root space decomposition
$V=\oplus_\a V_\a$. If each $V_\a^{\circ}=\{x \in V_{\a} \colon \esc{x,x}_0=0\}$ satisfies that $V_\a^\circ$ is not $\bot$-closed or $V_\a/V_\a^\circ$ is
infinite-dimensional, then $\F$ is simultaneously orthogonalizable.
\end{enumerate}
\end{theorem}
\begin{proof}
For proving item \ref{ogacem1}, we proceed analogously to the proof of item \ref{atupoi1} of Proposition \ref{atupoi}. So we have the existence of self-adjoint maps
$T_i\colon V\to V$ such that $\esc{x,y}_i=\esc{T_i(x),y}_0$ for any
$x,y\in V$ and any $i \in I$.
Now we prove item \ref{ogacem2},  assume that $\F$ is simultaneously orthogonalizable. Let $B=\{v_j\}$ be a basis of $V$ with $\esc{v_j,v_k}_i=0$ for $j\ne k$ and any $i\in I\cup\{0\}$. 
If we write $T_i(v_j)=\sum_k a_{ij}^k v_k$ we have 
$$\esc{T_i(v_j)-a_{ij}^j v_j, v_k}_0=\esc{T_i(v_j), v_k}_0-a_{ij}^j\esc{ v_j, v_k}_0 = \esc{v_j, v_k}_i=0 \hbox{ if $k\ne j$, }$$
$$\esc{T_i(v_j)-a_{ij}^j v_j, v_j}_0= \sum_{q\ne j}a_{ij}^q\esc{v_q,v_j}_0=0.$$ And since $\esc{\cdot,\cdot}_0$ is nondegenerate, then
$T_i(v_j)\in \K v_j$ for arbitrary 
$i, j \in I$. Thus each self-adjoint operator $T_i$ is diagonalizable in the basis $B$. 
Reciprocally, assume that for any $i\in I$ we have that each $T_i$ is diagonalizable relative to a certain orthogonal basis $B=\{v_j\}$ with respect to $\esc{\cdot,\!\cdot}_0$. 

Thus $T_i(v_j)\in \K v_j$ and 
we can write $T_i(v_j)=a_{ij}v_j$ for some $a_{ij}\in \K$.
So $\F$ is simultaneously orthogonalizable in $B$ since for any $i,j,k$ with $j\ne k$ we have: 
$$\esc{v_j,v_k}_i=\esc{T_i(v_j),v_k}_0=a_{ij}\esc{v_j,v_k}_0=0.$$ To prove item \ref{short} we only need to show that if $\T:=\{T_i\}_{i\in I}$ is \sd\ then there is an orthogonal basis of $V$ relative to $\esc{\cdot,\cdot}_0$ which diagonalizes the family $\T$. For this purpose, applying Proposition \ref{palmera} we have that $V=\oPerp_{\alpha}^{\esc{\cdot,\cdot}_0}V_\a$ is an orthogonal direct sum of root spaces. Moreover, since 
$$\esc{x,y}_i=\esc{T_i(x),y}_0=\a(T_i)\esc{x,y}_0,$$ so if $\esc{x,y}_0=0 $ then $\esc{x,y}_i=0$ for every $i \in I$ and therefore $V=\operpf{\a}V_\a$. Now, if we can find an orthogonal basis (relative to $\esc{\cdot,\cdot}_0$) in each $V_\a$, joining together all those bases  we get an orthogonal basis of $V$ relative to $\esc{\cdot,\cdot}_0$ which diagonalizes $\T$. Note that the restriction of $\esc{\cdot,\cdot}_0$ to each $V_\a$ is not alternate since on the contrary we would have $\esc{x,x}_0=0$ for any $x\in V_\a$ hence $\esc{\cdot,\cdot}_0\vert_{V_\a}=0$ which would imply the contradiction $V_\a\subset\rad(\esc{\cdot,\cdot}_0)=0$. Hence applying \cite[Chapter Two, Corollary 2, p. 65]{gross},
we get that there are orthogonal bases $B_{\a}$ relative to $\esc{\cdot,\cdot}_0$  in each $V_\a$. Now, for any different elements $x, y \in B_\a$  we have $$\esc{x,y}_i
=\esc{T_i(x),y}_0=\a(T_i)\esc{x,y}_0=0.$$ Next if we consider $B=\sqcup_\a B_{\a}$, then $B$ is an orthogonal basis of $V$ relative to any inner product of the family $\F$. 

Finally we show item \ref{short2}. We apply item \ref{ogacem2} for proving that when $\F$ is \so , then $\{T_i\}_{i\in I}$ is \sd. Now assume that $\{T_i\}_{i\in I}$ is \sd. Applying Proposition \ref{palmera} and analogously to the proof of item \ref{short}  we have that $V=\oPerp_{\a}^{\F} V_\a$. So the existence of an orthogonal basis relative to all the inner products will follow by noting that each orthogonal summand $V_\a$ has such a basis. For this, apply \cite[Chapter Two, Corollary 2, p. 65]{gross} taking into account that for each root $\a$ either
$V_\a^{\circ}=\{x \in V_{\a} \colon \esc{x,x}_0=0\}$  is not $\perp$-closed or $V_\a/ V_\a^{\circ}$ is infinite-dimensional.
\end{proof}

From the proof of Theorem \ref{ogacem} we can deduce the following corollary.

\begin{corollary} Suppose $\F=\{\esc{\cdot,\!\cdot}_i\}_{i\in I\cup\{0\}}$ is a nondegenerate family of inner products in a vector space $V$ over a field $\K$. If $\F$ is \so, then $V=\operpf{\a} V_\a$  where $\a$'s are the roots. Moreover, for any $i$ and $\a$, we get $\esc{\cdot,\!\cdot}_i\vert_{V_\a}=c_{i,\a}\esc{\cdot,\!\cdot}_0\vert_{V_\a}$ for suitable $c_{i,\a}\in \K$ and each $\esc{\cdot,\!\cdot}_i$ can be represented in block diagonal form where each block is the matrix of  $c_{i,\a}\esc{\cdot,\!\cdot}_0\vert_{V_\a}$ relative to some basis of $V_\a$.
\end{corollary}

We can go even further as the next corollary shows.

\begin{corollary}
Let $\F=\{\esc{\cdot,\!\cdot}_i\}_{i\in I\cup\{0\}}$ be a nondegenerate family of inner products in a vector space $V$ over a field $\K$ such that ${\ch}(\K)\ne 2$ and $\dim(V)$ is either finite or infinite countable. Then the following statements are equivalent:
\begin{enumerate}[label=(\roman*)]
    \item The family $\F$ is simultaneously orthogonalizable.
\item There exists an orthogonal basis $B$ such that for every $i \in I$ the matrix of the product  $\esc{\cdot,\cdot}_i$ relative to $B$ is a diagonal matrix.  
\end{enumerate} 
\end{corollary}
\begin{proof} Suppose that $\F$ is simultaneously orthogonalizable. Consider the basis $B= \sqcup_{\a} B_\a$ where $B_{\a}$ is the orthogonal basis of each $V_{\a}$. Let $M_{0,{B_\a}}$ 
be the matrix of the restriction of the product $\esc{\cdot,\cdot}_0$ to $V_\a$ relative to $B_{\a}$. Then, for each $i \in I$, we have 
$M_{i,B}:=M_B(\esc{\cdot,\cdot}_i)= {\rm {diag} }(\a(T_i)M_{0,{B_\a}}: \a \in \Phi)$, being $\Phi$ as in Proposition \ref{palmera}.
The converse follows immediately. 
\end{proof}

\subsection{Simultaneous orthogonalization of a degenerate family of inner products}
Note that a degenerate family $\F$ of inner products on $V$ is one such that either all the inner products in $\F$ are degenerate or if there is  one
nondegenerate inner product, then the remaining  inner products are not partially continuous relative to the topology induced by the first one.  We will focus on a family of inner products in which all their inner products are degenerate. This is because if we consider a nondegenerate inner product $\esc{\cdot,\cdot}_0$ and another one $\esc{\cdot,\cdot}_1$ which is not partially continuous relative to the $\esc{\cdot,\cdot}_0$-topology, then $\esc{\cdot,\cdot}_0$ and $\esc{\cdot,\cdot}_1$ cannot be simultaneously orthogonalizable in view of Proposition \ref{atupoi} item \ref{atupoi2}.

In this subsection we give a step ahead with relation to the previous one. We will assume that all the inner products in a given family $\F$ are degenerate but there is some element $\escemp_0\in\F$ whose radical is contained in the radical of the other members of $\F$. We will see that still we can recover much of the philosophy in the subsection \ref{gallo}. 

We will start by introducing some topological tools. If $X$ and $B$ are topological spaces and $A$ is a set, then any surjective map $f\colon A\to B$ is continuous relative to
the initial topology on $A$, that is, that whose closed sets are $f^{-1}(F)$ where $F$ ranges in the class of closed subsets of $B$. Moreover, for any other continuous map $\varphi\colon A\to X$ such that
for any $a,a'\in A$:
$$f(a)=f(a')\Rightarrow \varphi(a)=\varphi(a')$$
there is a unique continuous map $\theta\colon B\to X$ such that $\theta f=\varphi$. This is the abc of initial topologies but allows us to give an interesting topology on a $\K$-vector space $V$ endowed with an inner product. In fact, we have the following definition.

\begin{definition}\label{calabacin} \rm
Let $(V,\esc{\cdot,\cdot})$ be an inner product $\K$-vector space. Let $\r:=\rad(\esc{\cdot,\cdot})$ and define the inner product $\esc{\cdot,\cdot}_{\r}\colon V/\r\times V/\r\to\K$ as 
\begin{equation}\label{azorrac}
\esc{v+\r,w+\r}_{\r}:=\esc{v,w}  \text{ for } v,w \in V.
\end{equation}
Since $(V/\r,\esc{\cdot,\cdot}_{\r})$ is nondegenerate, we can consider the $\esc{\cdot,\cdot}_{\r}$-topology of $V/\r$ and consequently the initial topology on $V$ induced by the canonical projection $p\colon V\to V/\r$. We will call this topology the \emph{$\esc{\cdot,\cdot}$-topology} of $V$ by extension of the nondegenerate case.

Note that $\esc{\cdot,\cdot}$-topology of $V$ is the smallest topology such that $p$ is continuous.

\end{definition}

\begin{lemma}\label{tomate}
 Let $(V,\esc{\cdot,\cdot})$ be an inner product $\K$-vector space and topologize $V$ with the $\esc{\cdot,\cdot}$-topology. Then, $\esc{\cdot,\cdot}$ is partially continuous. 
 \end{lemma}

 \begin{proof}
Consider $a \in V$ and define $f_a: V \to \K$ as $f_a = \esc{a,\_}$ and $\tilde{f}_{a+\r}: V/\r \to \K$ as $\tilde{f}_{a+\r}(x+\r)=\esc{a+\r,x+\r}_{\r}=\esc{a,x}$ where $\r=\rad(\esc{\cdot,\cdot})$ . Observe that in fact $f_a$ is the composition $V \xrightarrow{p} V/\r \xrightarrow{\tilde{f}_{a+\r}} \K$. Hence in order to see that $f_a$ is continuous, it suffices to check that $\tilde{f}_{a+\r}$ is continuous. Since $V/\r$ is nondegenerate this is equivalent to find an adjoint map $(\tilde{f}_{a+\r})^{\sharp}: \K \to V/\r$. Define linearly $(\tilde{f}_{a+\r})^{\sharp}(1)=a+\r$, so for $x \in V$ and $\lambda \in \K$ we have:
$$ \esc{\tilde{f}_{a+\r}(x+\r),\lambda}_{\K}=\lambda \esc{a,x}=\esc{x,\lambda a}=\esc{x+\r_,\lambda (a+\r)}_{\r}=\esc{x+\r,(\tilde{f}_{a+\r})^{\sharp}(\lambda)}_{\r}.$$
\end{proof}

Before going on, consider an inner product $\K$-vector space $(V,\esc{\cdot,\cdot})$ with $\r=\rad{(\esc{\cdot,\cdot})}$. 
Then the quotient space $V/\r$ is nondegenerate relative to $\esc{\cdot,\cdot}_{\r}$, as defined in (\ref{azorrac}), and we fix our attention on it. If we choose a subspace $W$ of $V$ with $V=\r\oPerp W$, then there is a canonical vector space isomorphism $W\cong V/\r$. So we can consider the inner product $\K$-vector space $(W,\esc{\cdot,\cdot}\vert_W)$. 

If $(V_i,\esc{\cdot,\cdot}_i)_{i=1,2}$ are nondegenerate inner product spaces over the same field, a linear isomorphism $T\colon V_1\to V_2$ is said to be an {\em isometry} if $\esc{T(x),T(y)}_2=\esc{x,y}_1$ for arbitrary elements 
$x,y\in V_1$. If $T$ is an isometry, then $T$ is 
continuous (relative to the $\escemp_i$-topologies) being its adjoint $T^\sharp=T^{-1}$. Taking this into account, the above canonical isomorphism is an isometry $$(W,\esc{\cdot,\cdot}\vert_W)\cong (V/\r,\esc{\cdot,\cdot}_{\r}).$$
Note that the inverse of the above isometry is an isometry 
\begin{equation}\label{sartahcram}
\Omega\colon(V/\r,\esc{\cdot,\cdot}_{\r})\cong (W,\esc{\cdot,\cdot}\vert_W),\end{equation}
which will be useful in the sequel. 

\begin{lemma}\label{marcha} In the previous setting, 
the canonical inclusion $\ii \colon (W,\esc{\cdot,\cdot}\vert_W)\to V$ is continuous relative to the $\esc{\cdot,\cdot}$-topology of $V$. Moreover, the linear map $\ii \Omega: V/\r \to V$ is a continuous monomorfism.
\end{lemma}
\begin{proof}
Let $p\colon V\to (V/\r,\esc{\cdot,\cdot}_{\r})$ be the canonical projection. The open neighborhoods of $0$ in $V$ are of the form 
$p^{-1}(\cap_1^n (x_i+\r)^\bot)$ where $x_1,\dots,x_n$ is a finite collection of elements in $V$. For any $i$ write $x_i=r_i+w_i$ where $r_i\in\r$ and $w_i\in W$. Next we claim that  
$$p^{-1}(\cap_1^n (x_i+\r)^\bot)=\cap_1^n w_i^\bot \quad\text{ where } w_i^\bot=\{x\in V\colon \esc{x,w_i}=0\}.$$
If $v\in p^{-1}(\cap_1^n (x_i+\r)^\bot)$ then $p(v)\in \cap_1^n (x_i+\r)^\bot$ so that $\esc{p(v),p(x_i)}_{\r}=0$ for $i=1,\ldots,n$. But then formula \eqref{azorrac} gives $\esc{v,x_i}=0$ for any $i$ hence $\esc{v,w_i}=0$ so that 
$v\in\cap_i w_i^\bot$. The other inclusion is proved analogously. So far we have proved that the neighborhoods of $0$ in $V$ are of the form $\cap_1^n w_i^\bot$ where $w_1,\ldots,w_n$ is a finite collection of elements in $W$. Taking into account that 
$$\ii^{-1}(\cap_1^n w_i^\bot)=\cap_1^n w_i^{\bot_W}$$
where $w_i^{\bot_W}=\{x\in W\colon \esc{x,w_i}=0\}$, we get the continuity of $\ii$. Therefore, we have a continuous linear map  $V/\r\to V$ given by the composition $\ii\Omega$.
\end{proof}
We will call this map $\ii\Omega$ the \lq\lq backward gear\rq\rq\ from $V/\r$ to $V$.

\begin{proposition} \label{banana2} 
 Let $(V_i,\esc{\cdot,\cdot}_i)$ $(i=1,2)$ be two inner products $\K$-vector spacess  with $\r_i:=\rad(\esc{\cdot,\cdot}_i)$. We can write $V_i=\r_i \oPerp^{\esc{\cdot,\cdot}_i} W_i$ with $W_i$ a vector subspace of $V_i$. Let $T\colon V_1 \to V_2$ be a linear map such that $T(\r_1) \subset \r_2$ and $T(W_1) \subset W_2$. Then $T$ is continuous if and only if $T$ has adjoint.
 \end{proposition}

\begin{proof} Suppose that $T$ has an adjoint $T^{\sharp}$. The induced linear map  $S: V_1/\r_1 \to V_2/\r_2$ such that for $x \in V_1$, $S(x+\r_1)=T(x)+\r_2$ has an adjoint because for every $x \in V_1$, $y \in V_2$:
 $$\esc{S(x+\r_1),y+\r_2}_{\r_2}=\esc{T(x)+\r_2,y+\r_2}_{\r_2}=\esc{T(x),y}_2=$$ $$\esc{x,T^{\sharp}(y)}_1=\esc{x+\r_1,T^{\sharp}(y)+\r_1}_{\r_1}.$$
Then an adjoint of $S$ is $S^{\sharp} \colon V_2/\r_2 \to V_1/\r_1$ given by $S^{\sharp}(y+\r_2):=T^{\sharp}(y)+\r_1$. Thus $S$ is continuous.
 Next we prove that $T=\ii_2\Omega_2 S p_1$ where $\ii_2\Omega_2$ is defined as in Lemma \ref{marcha} on $V_2/\r_2$. Take an arbitrary $x\in V_1$, then 
$\ii_2\Omega_2 S p_1(x)=\ii_2\Omega_{2}(T(x)+\r_2)=\ii_2(z)$ for some $z\in W_2$ such that 
$z+\r_2=T(x)+\r_2$. Applying the hypothesis we have $z=T(x)$. Thus we have the commutativity of the following diagram
 \[
\begin{tikzcd}
 V_1/\r_1 \arrow[r,"S"]& V_2/\r_2 \arrow[d,"\ii_2\Omega_{2}"]\\
 V_1 \arrow[r,"T"'] \arrow[u,"p_1"]& V_2\\
\end{tikzcd}
\]

Observe that $T$ is continuous because $p_1, S$ and $\ii_2 \Omega_{2}$ are continuous (apply Lemma \ref{marcha}).

Conversely, assuming that $T$ is continuous, define $S:=p_2T \ii_1\Omega_{1}$ where $\ii_1\Omega_{1}$ is defined as in Lemma \ref{marcha} on $V_1/\r_1$.  Since $S$ is continuous there is an adjoint $S^\sharp\colon V_2/\r_2\to V_1/\r_1$ by \cite[\S IV, section 7, Theorem 1, p. 72]{Jacobson}). Let $x\in V_1$ and $y \in V_2$. If we write  $S^\sharp(y + \r_2)=z + \r_1$ for certain $z \in W_1$ then $\ii_1\Omega_1 (z+\r_1)=w \in W_1$ such that $w + \r_1=z + \r_1$. Now we will check that the map $\ii_1\Omega_{1}S^\sharp p_2\colon V_2\to V_1$ vanishing on $\r_2$ is an adjoint of $T$:
$$\esc{T(x),y}_2=\esc{T(x)+\r_2,y+\r_2}_{\r_2}=\esc{S(x+\r_1),y+\r_2}_{\r_2}=\esc{x+\r_1,S^\sharp(y+\r_2)}_{\r_1}=
$$
$$
\esc{x+\r_1,z+\r_1}_{\r_1}=\esc{x,z}_{1}=\esc{x,w}_{1}
=\esc{x, \ii_1\Omega_{1}S^\sharp(y+\r_2)}_1=\esc{x,\ii_1\Omega_1S^{\sharp}p_2 (y)}_{1}.$$
Therefore $T$ has an adjoint.
\end{proof}

Proposition \ref{banana2} is a generalization of the well-known principle asserting that in the context of nondegenerate spaces, a map with posses an adjoint map is automatically continuous (see \cite[\S IV, section 7, Theorem 1, p. 72]{Jacobson}).

\begin{remark}\label{gato}\rm  Assume that $\F$ is a family of inner products such that there exists $\esc{\cdot,\cdot}_0\in\F$ for which any element $\esc{\cdot,\cdot}\in\F$ is partially continuous relative to the $\esc{\cdot,\cdot}_0$-topology of $V$. Assume further that $\r_0:=\rad(\esc{\cdot,\cdot}_0)\subset \rad(\esc{\cdot,\cdot})$ for any $\esc{\cdot,\cdot}\in\F$. Consider next the new family $\F_{\r_0}$ of inner products $\esc{\cdot,\cdot}'$ on $V/\r_0$
defined as $\esc{x+\r_0,y+\r_0}':=\esc{x,y}$ for any $\esc{\cdot,\cdot}\in\F$. Then each element
$\esc{\cdot,\cdot}'$ is partially continuous relative to
$\esc{\cdot,\cdot}_{\r_{0}}$. To prove this, fix $x\in V$ and consider the map $g\colon V/\r_0\to\K$ such that $g(y +\r_0):=\esc{x+\r_0,y+\r_0}'$.
We have to prove that $g$ is continuous (more precisely $\esc{\cdot,\cdot}_{\r_{0}}$-continuous). But the map $f\colon V\to\K$ given by $f(y):=\esc{x,y}$ is continuous (that is, $\esc{\cdot,\cdot}_0$-continuous), and as $g$ is the composition $g=f\ii\Omega$ whence the continuity of $g$.
\[
\begin{tikzcd}[column sep=tiny, row sep=small]
V/\r_0\arrow[rr,"g"]\arrow[dr,"\ii\Omega"'] & &\K\\
& V\arrow[ur,"f"'] & 
\end{tikzcd}
\]
This allows to conclude that the new family $\F_{\r_0}$ has a nondegenerate element 
$\esc{\cdot,\cdot}_{\r_{0}}$ and any other $\esc{\cdot,\cdot}'\in\F_{\r_0}$ is partially continuous relative to $\esc{\cdot,\cdot}_{\r_{0}}$. 
\end{remark}

\begin{proposition}\label{berenjena}
Let $V$ be a $\K$-vector space endowed with two inner products $\esc{\cdot,\cdot}_i$ ($i=0,1$) such that $\rad(\esc{\cdot,\cdot}_0)\subset\rad(\esc{\cdot,\cdot}_1)$. Then the following assertions are equivalent:
\begin{enumerate}[label=(\roman*)]
    \item $\esc{\cdot,\cdot}_1$ is partially continuous relative to the $\esc{\cdot,\cdot}_0$-topology of $V$.
    \item \label{ber} There is a continuous linear map $T\colon V\to V$ (relative to the $\esc{\cdot,\cdot}_0$-topology) vanishing on $\rad(\esc{\cdot,\cdot}_0)$  such that $\esc{x,y}_1=\esc{T(x),y}_0$ for any $x,y\in V$.
\end{enumerate}

\rm\bigskip
\end{proposition}

\begin{proof}
Consider $\r_0=\rad(\esc{\cdot,\cdot}_0)$,  $\r_1=\rad(\esc{\cdot,\cdot}_1)$, the inner product $\esc{\cdot,\cdot}_{\r_{0}}$ as \eqref{azorrac}, the $\esc{\cdot,\cdot}_{\r_{0}}$-topology of $V/\r_0$ and the decomposition $V=\r_0\oPerp^{\esc{\cdot,\cdot}_0} W$ for certain subspace $W$ of $V$. Let $p\colon V\to V/\r_0$ that gives the initial topology on $V$. Define now the inner product $\esc{\cdot,\cdot}_1'\colon V/\r_0\times V/\r_0 \to \K$ as $\esc{x+\r_0,y+\r_0}_1':=\esc{x,y}_1$. It is well defined because $\r_0\subseteq\r_1$.
As $\esc{\cdot,\cdot}_1$ is partially continuous then for any $x\in V$ the linear map $\esc{x,\_}_1\colon V\to\K$ is continuous.  Since $\r_0\subset\ker(\esc{x,\_}_1)$
there is a unique linear map $\esc{x+\r_0,\_}'_1\colon V/\r_0\to\K$ such that the diagram below is commutative
\[
\begin{tikzcd}
V \arrow[r,"p"]\arrow[dr,"\esc{x,\_}_1"']& V/\r_0\arrow[d, dashed, "\esc{x+\r_0,\_}'_1"]\\
   & \K\\
\end{tikzcd}
\]
and the continuity of $\esc{x+\r_0,\_}'_1$ is automatic from the universal property.

Now applying Proposition \ref{atupoi} item \ref{atupoi1} there exists a continuous linear map $S: V/\r_0 \to V/\r_0$ (continuity relative to the $\esc{\cdot,\cdot}_{\r_{0}}$-topology of $V/\r_0$), such that $\esc{x+\r_0,\_}'_1=\esc{S(x+\r_0),\_}_{\r_{0}}$. Define $T\colon V\to V$ such that $T=\ii\Omega Sp$ having the following commutative diagram:

\[
\begin{tikzcd}
 V/\r_0 \arrow[r,"S"]& V/\r_0\arrow[d,"\ii\Omega"]\\
 V \arrow[r,"T"] \arrow[u,"p"]& V\\
\end{tikzcd}
\]

So $T(x):= \ii \Omega(S(x+\r_0))$. Observe that $T$ is linear and continuous because $p, S$ and $\ii \Omega$ are linear and continuous (apply Lemma \ref{marcha}). In fact, if we write $S(x+\r_0) = w + \r_0$ for a unique $w \in W$, then we have that $S(x+\r_0)=p(w)$ and $T(x)=w$. Moreover for $x,z \in V$ we get
$$
\esc{x,z}_1\!=\!\esc{x+\r_0, z+\r_0}'_1\!=\!\esc{S(x+\r_0), z+\r_0}_{\r_{0}}=\! \esc{p(w), z+\r_0}_{\r_{0}}=\!\esc{w,z}_0\!=\esc{T(x),z}_0.
$$
For the converse, (ii) implies (i) by Lemma \ref{tomate}. \end{proof}

\begin{remark}\label{cachondo} \rm
Observe that the map $T$ defined in the proof of Proposition \ref{berenjena} satisfies $T(V)\subset W$.
\end{remark}

\begin{theorem} \label{oepem}
Let $\F=\{\esc{\cdot,\cdot}_i\}_{i\in I\cup \{0\}}$ be a family of inner products in the $\K$-vector space $V$ such that $\r_0:=\rad(\esc{\cdot,\cdot}_0)\subseteq \rad(\esc{\cdot,\cdot}_i)$ for any $i\in I$ and each $\esc{\cdot,\cdot}_i$ is partially continuous relative to the $\esc{\cdot,\cdot}_0$-topology of $V$. Then:
\begin{enumerate}[label=(\roman*)]
    \item\label{oepem1} 
$\F$ is \so\  if and only there exist an orthogonal basis $B=\{v_j\}$ of $(V,\esc{\cdot , \cdot}_0)$ and a collection of $\esc{\cdot,\cdot}_0$-self-adjoint continuous linear maps $T_i\colon V \to V$ for every $i \in I \cup \{0\}$ such that $\{T_i\}$ is \sd\ and for any $i$ we have $T_i(\r_0)=0$, $\esc{x,y}_i=\esc{T_i(x),y}_0$ for any $x,y\in V$. 
\item \label{oepem2} In case ${\ch}(\K)\ne 2$ and $\dim(V)$ is either finite or infinite countable, the family $\F$ is simultaneously orthogonalizable if and only if $\{T_i\}_{i\in I}$ (as in item \ref{oepem1}) is \sd.
\end{enumerate}
\end{theorem}
\begin{proof}

In order to prove item \ref{oepem1} suppose first that $\F$ is \so. Then there exists a basis $B=\{v_j\}_{j \in J}$ such that $\esc{v_j, v_k}_i=0$ for $i \in I \cup \{0\}$ and $j\ne k$. Consider $\F_{\r_0}=\{\esc{\cdot , \cdot}'_i\}_{i \in I \cup\{0\}}$ where $\esc{\cdot,\cdot}'_i$ is the inner product defined on the quotient $V/\r_0$ by 
$\esc{x+\r_0, y +\r_0}'_i:=\esc{x,y}_i$ (note that $\esc{\cdot,\cdot}'_0=\esc{\cdot, \cdot}_{\r_{0}}$). Observe that since $\F$ is \so, then $\F_{\r_0}$ is \so. By Remark \ref{gato} and Proposition \ref{berenjena} item \ref{ber} we have a collection of continuous linear maps $T_i \colon V \to V$ (relative to $\esc{\cdot , \cdot}'_0$) vanishing on $\rad(\esc{\cdot,\cdot}_0)$  such that $\esc{x,y}_i=\esc{T_i(x),y}_0$ for any $x,y\in V$ and $i \in I$. Moreover, we know that $T_i= \ii \Omega S_i p$ where
$S_i \colon V/\r_0 \to V/\r_0$  verifies $\esc{x+\r_0, y +\r_0}'_i=\esc{S_i(x+\r_0),y+\r_0}'_0$, $V \xrightarrow{p} V/\r_0$ and $V/\r_0 \xrightarrow{\ii \Omega} V$ as in the proof of Proposition \ref{berenjena}. Since $\F_{\r_0}$ is \sd\ by Theorem \ref{ogacem} item \ref{ogacem2} $S_i$ is simultaneously diagonalizable. So for $j \in J$ there exist $\lambda_{ij} \in \K$ such that $S_i(v_j + \r_0)=\lambda_{ij}(v_j+\r_0)$. Now we prove that $T_i$ is \sd. Indeed, take $v_j \in W$ we have that $T_i(v_j)=\ii \Omega S_i p(v_j)= \ii \Omega S_i (v_j + \r_0)= \ii \Omega (\lambda_{ij}(v_j +\r_0))=\lambda_{ij} \ii \Omega(v_j+\r_0)= \lambda_{ij} \ii(v_j)=\lambda_{ij} v_j$.

For the converse, suppose that there exists a basis $B=\{v_j\}$ such that $\esc{v_j,v_k}_0=0$ if $j \neq k$ and $T_i(v_j)=a_{ij}v_j$ with $a_{ij} \in \K$. Now, for every $i \in I$ we get $\esc{v_j,v_k}_i=\esc{T_i(v_j),v_k}_0=a_{ij}\esc{v_j,v_k}_0=0$ if $j \neq k$. 

Next we prove \ref{oepem2}. The nontrivial implication is as follows. Assume that $\{T_i\}_{i\in I}$ is \sd. Consider $\{S_i\}_{i\in I}$ given by $S_i \colon V/\r_0 \to V/\r_0$ defined as $S_i(x+\r_0)=T_i(x) + \r_0$. Since $\ch(\K) \neq 2$ and $\dim(V/\r_0)$ is either finite or infinite countable then, applying Theorem \ref{ogacem} item \ref{short}, we get that $\{\esc{\cdot,\cdot}'_i\}_{i \in I}$ is \so. Consequently we have a basis $\{x_j + \r_0\}_{j \in J}$ of $V/\r_0$ that orthogonalizes simultaneously. Let $\ii \Omega \colon V/\r_0 \to V$ be as in Lemma \ref{marcha}, that is, for $x \in V$ we have $\ii \Omega(x+\r_0)=w$ with $x+\r_0=w+\r_0$ where $w \in W$ and $V= W \oPerp \r_0$. Suppose $\ii \Omega (x_j + \r_0)=t_j$. Observe that $t_j \in W$ for every $j \in J$. Next, $\{t_j\}_{j \in J}$ is a basis of $W$ because $\Omega$ is an isomorphism. For our purposes, consider $B= \{t_j\}_{j \in J} \sqcup \{r_k\}_{k \in K}$ where  $\{r_k\}$ is a basis of $\r_0$. For $j \neq k$, we finally have
$$\esc{t_j,t_k}_i=\esc{t_j+\r_0,t_k+\r_0}'_i=\esc{x_j+\r_0,x_k+\r_0}'_i=0,$$
providing thus a basis $B$ of $V$ which simultaneously orthogonalizes the family $\F$.
\end{proof}

\begin{remark}\label{guanga}\rm
The $T_i$'s defined in the proof of Theorem \ref{oepem} item \ref{oepem1} may be chosen in such a way that $T_i(V)\subset W$ for any $i$.
\end{remark}

Analogously to the nondegenerate case we can conclude:

\begin{corollary} Suppose that $\F=\{\esc{\cdot,\!\cdot}_i\}_{i\in I\cup\{0\}}$ is a family of inner products in a vector space $V$ over a field $\K$ such that $\r_0:=\rad(\esc{\cdot,\cdot}_0)\subseteq \rad(\esc{\cdot,\cdot}_i)$ for any $i\in I$ and each $\esc{\cdot,\cdot}_i$ is partially continuous relative to the $\esc{\cdot,\cdot}_0$-topology of $V$. If $\F$ is \so, then $V=\operpf{\a} V_\a$ where $\a$'s are the roots for the family of endomorphisms $\{T_i\}$ given by Theorem $\ref{oepem}$ item $\ref{oepem1}$.  Moreover, for any $i$ and $\a$, we get $\esc{\cdot,\!\cdot}_i\vert_{V_\a}=c_{i,\a}\esc{\cdot,\!\cdot}_0\vert_{V_\a}$ for suitable $c_{i,\a}\in \K$ and each $\esc{\cdot,\!\cdot}_i$ can be represented in block diagonal form where each block is the matrix of  $c_{i,\a}\esc{\cdot,\!\cdot}_0\vert_{V_\a}$ relative to some basis of $V_\a$.
\end{corollary}

\begin{proof}
First we can write $V=W \oPerp \r_0$ for certain subspace $W \subset V$. We know that there exists a family of endomorphisms $\{T_{i}\}_{i \in I}$ in the conditions of Theorem \ref{oepem} item \ref{oepem1}. By Proposition \ref{palmera} we have  $V=\bigoplus_{\a \in \Phi} V_{\a}$. Next we prove that for each $\alpha$, $V_{\alpha} = (V_{\a} \cap \r_0) \oplus (V_{\a} \cap W)$. Let $x=w+r \in V_{\a}$ where $w \in W, \, r \in \r_0$. By Remark \ref{guanga}, $T_i(V)\subset W $ for $i \in I$. Furthermore
$T_i(x)=T_i(w)$ since $T_i(r)=0$. But $T_i(x)=\a(T_i)x=\a(T_i)w+\a(T_i)r=\a(T_i)w$ which means $T_i(w)=\a(T_i)w$ that is, $w \in V_{\a} \cap W$. Consequently $r=x-w \in V_{\a} \cap \r_0$, so $V_{\alpha} \subseteq (V_{\a} \cap \r_0) \oplus (V_{\a} \cap W)$. The  other containment is trivial. Moreover, $V_{\alpha} = (V_{\a} \cap \r_0) \operpf{\a} (V_{\a} \cap W)$  because $\rad(\esc{\cdot,\cdot}_0)\subseteq \rad(\esc{\cdot,\cdot}_i)$. 
Secondly, we check $V_{\a} \perp^{\F} V_{\beta}$ for $\a \neq \beta$. Indeed it suffices to prove that $(V_{\a} \cap W)  \perp^{\F} (V_{\beta} \cap W)$. For this aim, take $x \in V_{\a} \cap W$ and $y \in V_{\beta} \cap W$: $\esc{x,y}_i=\esc{T_i(x),y}_0=\a(T_i)\esc{x,y}_0$ and $\esc{x,y}_i=\esc{y,x}_i=\esc{T_i(y),x}_0=\beta(T_i)\esc{x,y}_0$, so then $\a(T_i)\esc{x,y}_0=\beta(T_i)\esc{x,y}_0$. Since $\a \neq \beta$, there exists $T_k$ such that $\a(T_k)\neq \beta(T_k)$, whence $\esc{x,y}_0=0$ as desired. Finally, as $\esc{x,y}_i=\esc{T_i(x),y}_0=\a(T_i)\esc{x,y}_0$ for any $x,\, y \in V_{\a}$ then $\esc{\cdot,\!\cdot}_i\vert_{V_\a}=c_{i,\a}\esc{\cdot,\!\cdot}_0\vert_{V_\a}$ with $c_{i,\a}=\a(T_i)$. Taking this into account each $\esc{\cdot,\!\cdot}_i$ can be represented in block diagonal form where each block is the matrix of  $c_{i,\a}\esc{\cdot,\!\cdot}_0\vert_{V_\a}$ relative to some basis of $V_\a$.
\end{proof}

\section{Different constructions of a new family of inner products with a nondegenerate element}

We have seen in Remark \ref{fragel} that a necessary condition for a family of inner products $\F$ on $V$ to be \so\ is that we can add (at most) one inner product so as to get a new family $\F'$ which is nondegenerate. Furthermore if $B$ is a basis of $V$ orthogonalizing
$\F$, the new family $\F'$ can be constructed in such a way that the same basis $B$ orthogonalizes it.

In this section we consider several constructions whose goal is to enlarge a given family $\F$ of inner products so as to get a new family $\F'$ whose simultaneous orthogonalization is essentially the same as that of $\F$ and
such that $\F'$ is better behaved than the original $\F$. The optimum than we can expect is to get a nondegenerate inner product in $\F'$ while possibly all the inner products in $\F$ are degenerate. Also we investigate some properties that occur automatically in the finite-dimensional case, but that deserve certain attention in the infinite-dimensional case.

\subsection{Construction by adding a linear combination.}

Under certain conditions, it is possible to construct a new family containing a nondegenerate inner product from a given one (in which possibly all the inner products are degenerate). Moreover, if the initial family is \so, then the new one is also \so.

\begin{definition}\rm  Let $\F=\{\esc{\cdot,\cdot}_i\}_{i \in I}$ be a family of inner products in the $\K$-vector space $V$. We define the {\it radical of the family} $\F$ by $\rad(\F):=\cap_{i\in I}\rad(\esc{\cdot,\cdot}_i)$. For any $F\in\PF(I)$ we define the subspace $$R_F:=\cap_{i\in F}\rad(\esc{\cdot,\cdot}_i).$$
We will say that $\F$ satisfies the {\em descending chain condition on radicals} if it satisfies the descending chain condition on subspaces $R_F$. More precisely: for any sequence $F_1, F_2,\ldots $ of finite subsets of $I$
such that $R_{F_1}\supset R_{F_2}\supset\cdots\supset R_{F_k}\supset\cdots$,
there is some $n$ such that $R_{F_n}=R_{F_m}$ for any $m\ge n$.
\end{definition}

It is easy to see that $\F$ satisfies the descending chain condition if and only if any collection of sets $R_F$ has a minimum element (relative to the inclusion). Of course, if $V$ happens to be finite-dimensional automatically $\F$ satisfies the descending chain condition on radicals.
An example of $\F$ which does not satisfy the descending chain conditions on radicals is $\F=\{\esc{\cdot,\cdot}_i\}_{i\in\N^*}$ where $V$ is a  countably-dimensional $\K$-vector space and there is an (simultaneous) orthogonal basis $\{v_j\}$ such that
for any $i$ we have $\esc{\cdot,\cdot}_i$ given by 
$\esc{v_j,v_j}_i=1$ for $1\le j\le i$ and 
$\esc{v_j,v_j}_i=0$ for $j>i$. This $\F$ has  $\rad(\F)=0$.

\medskip
Although the following result is well-known in the literature,  we give here a slightly more general version than the one we can find in \cite[\S X, section 14, Proposition 3, p. 248]{Jacobson}.

\begin{lemma}\label{retorta}
Let $V$ be a $\K$-vector space of finite dimension.
Let $J$ be a set of cardinal $\omega$ and $\K$ be a field of cardinal $>\omega$. Let $\{l_i\}_{i\in J}$ be a family of proper subspaces of $V$ indexed by $J$. Then $V\not\subset \cup_{i\in J}l_i$. 
\end{lemma}
\begin{proof}
Let $n$ be the dimension of $V$. For $n=1$ this is trivial. Suppose that the assertion is true for any $k<n$.
So any subspace of dimension $<n$ is not a union (indexed by $J$) of proper subspaces of $V$. Assume on the contrary that $V$ is a $\K$-vector space of dimension $n$ and that it is an union of proper subspaces $V\subset \cup_{i\in J}l_i$.  Let $L=\{l_i\}_{i\in J}$. Let $h$ be a hyperplane not in $L$. It exists because $\K$ has cardinal  $>\omega$ and $L$ is a collection whose cardinal is $\leq \omega$ (so $V$ contains many more hyperplanes). Then $h=h\cap V \subset \cup_{i\in J} (h\cap l_i)$. But for any $i$ we have $h\cap l_i$ is a proper subspace of $h$ (on the contrary $h\subset l_i$ and this would imply $h=l_i$). So we have that $h$ (which has dimension $n-1$) is a union (indexed by $J$) of proper subspaces. This is a contradiction.
\end{proof}

Note that if $J$ has finite cardinal $\o$ and $\rm{card}(\K)>\o$, Lemma \ref{retorta} applies. It also applies if $J$ is countable and $\K$ uncountable.

\begin{corollary}\label{krop}
Let $\K$ be a field, $n\in \N$, $n\ge 1$
and $L_n=\{(a_{1i},\ldots, a_{ni})\in\K^n \colon i\in J \}$ with $\rm{card}(\K) > \rm{card}(J)$. Then there are elements 
$x_1,\ldots,x_n\in\K$ such that for any $i$ we have $x_1 a_{1i}+\cdots+x_n a_{ni}\ne 0$.
\end{corollary}
\begin{proof}
Assume the contrary, so we can write that for any $(x_1,\ldots,x_n)\in\K^n$ we have
$a_{1j}x_1+\cdots+a_{nj}x_n=0$  for some $j$. 
In geometrical terms this means that $\K^n$ is contained in a union of hyperplanes indexed by $J$ which is not possible by Lemma \ref{retorta}.
\end{proof}

\begin{proposition}
Let $V$ be a $\K$-vector space and
 $\F=\{\esc{\cdot,\cdot}_i\}_{i\in I}$ be a \so\ family of inner products of $V$ satisfying the descending chain condition on radicals. Then 
there is a finite subset $F\subset I$ such that \begin{equation}\label{merenguito}  
\bigcap_{i\in F}\rad(\esc{\cdot,\cdot}_i)=\rad(\F).
\end{equation}
Furthermore if $\rad(\F)=0$ and $\rm{card}(\K) > \dim(V)$, there exists a linear combination  of elements of $\F$ which is nondegenerate.
\end{proposition}
\begin{proof}
If $\rm{card}(I)$ is finite the equality \eqref{merenguito} holds trivially.
So assume that $I$ is an infinite set. Denote $J=\PF(I)$ and consider the collection
$\{R_F\}_{F\in J}$ which by our hypothesis has a minimum element $R_{F_0}$.
On the one hand we have $\rad(\F)\subset R_{F_0}$. On the other hand 
$\forall F\in J$ we have  $R_F\supset R_{F_0}$. Thus 
$$R_{F_0}\supset\rad(\F) \supset \bigcap_{F\in J} R_F\supset R_{F_0}.$$
Now we prove the second part. For some set of inner products (which can be chosen linearly independent) we have $0=\rad(\esc{\cdot,\cdot}_1)\cap\cdots\cap\rad(\esc{\cdot,\cdot}_n)$ (reordering if necessary). If the simultaneous orthogonal basis is $\{v_k\}$ and $\esc{v_k,v_k}_j=a_{jk}$.
The matrix of $\esc{\cdot,\cdot}_j$ is $\text{diag}(a_{ji})_{i \in H}$ where $\rm{card}(H)= \dim(V)$, then applying Corollary \ref{krop}
there are scalars $x_1,\ldots, x_n\in\K$ such that 
$x_1\esc{\cdot,\cdot}_1+\cdots+x_n\esc{\cdot,\cdot}_n$ is nondegenerate.
\end{proof}

Note that in the hypothesis of the previous proposition we can construct a new family $\F'=\F\cup\{x_1\esc{\cdot,\cdot}_1+\cdots+x_n\esc{\cdot,\cdot}_n\}$ which contains a nondegenerate element and induces the same simultaneous orthogonalization as $\F$. Furthermore the new family is obtained by adding at most one element.
Moreover, $n$ can be chosen minimum.

\subsection{Scalar extension of inner products}\label{periquito}
The aim of this subsection is to study the behaviour of nondegenerate families of
inner products under scalar extension.

If $\esc{\cdot,\cdot}\colon V\times V\to\K$ is an inner product and $\FC\supset\K$ a field extension, then we can define on $V_\FC:= V \otimes \FC$ an inner product 
$\esc{\cdot,\cdot}_\FC\colon V_\FC\times V_\FC\to\FC$ such that 
$\esc{v\otimes\l,w\otimes\m}_{\FC}=\l\m\esc{v,w}$. If we take a (possibly infinite) basis $B=\{v_i\}$ of $V$, then $B\otimes 1:=\{v_i\otimes 1\}$ is a basis of $V_\FC$ as $\FC$-vector space. In the particular case in which $V$ is finite-dimensional, 
the matrix  of $\esc{\cdot,\cdot}_\FC$ relative to the basis $B\otimes 1$ coincides 
with the matrix of $\esc{\cdot,\cdot}$ (relative to $B$) up to the canonical identification of $V$ with $V \otimes \K$.
Consequently 
\begin{equation}\label{pepino}
\rad(\esc{\cdot,\cdot}_{\FC}) \cong \rad(\esc{\cdot,\cdot})\otimes\FC
\end{equation}

\noindent
and in particular nondegeneracy of $\esc{\cdot,\cdot}$ implies that of $\esc{\cdot,\cdot}_\FC$. We can prove the above isomorphism  for arbitrary $V$, taking a basis $\{v_i\}$ of $V$ and a basis $\{f_j\}$ of the $\K$-vector space $\FC$. If $\esc{z,V}=0$ then automatically $\esc{z\oo 1,V_\FC}_{\FC}=0$ so that the map
$\rad(\esc{\cdot,\cdot})\otimes\FC\to V_\FC$ such that $z\oo 1\mapsto z\oo1$ is a monomorphism with image contained in $\rad(\esc{\cdot,\cdot}_\FC)$. To prove that the image of the above map is $\rad(\esc{\cdot,\cdot})_\FC$, let
 $c_{ij}\in\K$ be scalars and take an element $z=\sum_{ij}c_{ij}v_i\otimes f_j\in V_\FC$ such that $\esc{z,V_\FC}_{\FC}=0$. Then $\esc{z,u\otimes 1}_\FC=0$ for any $u\in V$. Thus $0=\sum_{ij}f_jc_{ij}\esc{v_i,u} $ which implies $\sum_i c_{ij}\esc{v_i,u}=0$ for any $j$ and $u \in V$. Whence $\sum_i c_{ij}v_i\in \rad(\esc{\cdot,\cdot})$ for any $j$. 
 Defining $z_j:=\sum_i c_{ij}v_i\in \rad(\esc{\cdot,\cdot})$
 we have $z=\sum_j z_j\oo f_j$ hence $z\in\rad(\esc{\cdot,\cdot})\otimes\FC$.
 
\begin{definition}\rm
 Consider a field extension $\FC\supset\K$ and a linear map $T\colon V_1\to V_2$ with $V_1$ and $V_2$ two $\K$-vector spaces. We will denote by $T_\FC:=T\oo 1$ the linear map $V_{1\FC}\to V_{2\FC}$ such that $T_\FC(v\oo\lambda)=T(v)\oo\lambda$ for any $v\in V_1$ and $\lambda\in\FC$. 
\end{definition}
 
Observe that if $T$ has an adjoint $S$ then $T_\FC$ has also an adjoint $S_\FC$.
 
 \begin{remark}\label{tapita}
\rm If $V_i$ is provided with
 a nondegenerate inner product $\esc{\cdot,\cdot}_i$ $(i=1,2)$ and $T$ is a continuous linear map 
 $T\colon V_1\to V_2$ (relative to the $\esc{\cdot,\cdot}_i$-topology of $V_i$ for $i=1,2$), then $T_\FC$ is also continuous (relative to the $\esc{\cdot,\cdot}_{i\FC}$-topology of $V_{i\FC}$ for $i=1,2$) since continuity in this case is characterized by the existence of an unique adjoint and we have $(T_\FC)^\sharp=(T\oo 1)^\sharp=T^\sharp\oo 1=(T^\sharp)_\FC$ (see \cite[\S IV, section 7, Theorem 1, p. 72]{Jacobson}).
 \end{remark}

\begin{proposition}\label{sotreumsus}
For $i=1,2$ let $(V_i,\esc{\cdot,\cdot}_i)$ be inner product $\K$-vector spaces with radical $\r_i$, and $T\colon V_1\to V_2$ be a continuous linear map (relative to the $\esc{\cdot,\cdot}_i$-topology of $V_i$). Let $\FC\supset\K$ be a field extension. Assume that $T(\r_1)\subset\r_2$, $T(W_1)\subset W_2$ for some subspaces $W_i$ with $V_i= \r_i \oPerp W_i$, $i=1,2$. Then the linear map $T_\FC$ is continuous (relative to the $\esc{\cdot,\cdot}_{i\FC}$-topology of $V_{i\FC}$).
\end{proposition}
\begin{proof}
Applying Proposition \ref{banana2} there exists an adjoint $T^\sharp\colon V_2\to V_1$. Then it follows by definition that $T_\FC$ has also an adjoint which is $(T^\sharp)_\FC$. Then, again Proposition \ref{banana2} gives the continuity of
$T_\FC$ because $T_\FC(\rad(\escemp_{1\FC})=T(\r_1)\oo\FC\subset\r_2\oo\FC=\rad(\escemp_{2\FC})$ and 
$V_{i\FC}=\rad(\escemp_{i\FC})\oPerp W_{i\FC}$ for $i=1,2$. Furthermore 
$T_\FC(W_{1\FC})=T(W_1)\oo\FC\subset W_2\oo\FC=W_{2\FC}$.
\end{proof}

\begin{proposition}
For $i=0,1$ let $(V,\esc{\cdot,\cdot}_i)$  be two inner product $\K$-vector spaces  with $\rad(\esc{\cdot,\cdot}_0)\subset \rad(\esc{\cdot,\cdot}_1)$. Then if $\esc{\cdot,\cdot}_1$ is partially continuous relative to the $\esc{\cdot,\cdot}_0$-topology of $V$, for any field extension $\FC\supset\K$ the inner product 
$\esc{\cdot,\cdot}_{1\FC}$ is partially continuous relative to the $\esc{\cdot,\cdot}_{0\FC}$-topology of $V$.
\end{proposition}

\begin{proof}
Indeed, if $\esc{\cdot,\cdot}_1$ is partially continuous relative to the $\esc{\cdot,\cdot}_0$-topology, Proposition \ref{berenjena} implies that $\esc{x,y}_1=\esc{T(x),y}_0$ for some $T\colon V\to V$ linear, continuous and vanishing on $\rad(\esc{\cdot,\cdot}_0)$. Then  $T_\FC\colon V_\FC\to V_\FC$ is continuous by Proposition \ref{sotreumsus} (taking into account Remark \ref{cachondo}). Moreover, we have $$\esc{x\oo\l,y\oo\mu}_{1\FC}=\l\m \esc{x,y}_1=\l\m\esc{T(x),y}_0=\esc{T_\FC(x\oo\l),y\oo\m}_{0\FC}.$$ So $\esc{z,z'}_{1\FC}=\esc{T_\FC(z),z'}_{0\FC}$ for any $z,z'\in V_\FC$ and since $\esc{\cdot,\cdot}_{0\FC}$ is partially continuous by Lemma \ref{tomate} we conclude that $\esc{\cdot,\cdot}_{1\FC}$ is also partially continuous. 
\end{proof}

Let $\F$ be a family of inner products in a $\K$-vector space $V$.
Assume that: 

\begin{enumerate}[label=(\roman*)]
\item there is an element $\esc{\cdot,\cdot}_0\in\F$ whose radical $\r_0$ is contained in any $\rad(\esc{\cdot,\cdot})$ for arbitrary $\esc{\cdot,\cdot}\in\F$; 
\item each $\esc{\cdot,\cdot}$ is partially continuous relative to the $\esc{\cdot,\cdot}_0$-topology of $V$.
\end{enumerate} 

Then for any field extension $\FC\supset\K$, the extended family $$\F_\FC:=\{\esc{\cdot,\cdot}_\FC\colon \esc{\cdot,\cdot}\in\F\}$$
satisfies both properties: $\rad(\esc{\cdot,\cdot}_{0\FC})\subset \rad(\esc{\cdot,\cdot}_\FC)$ for any $\esc{\cdot,\cdot}_\FC\in\F_\FC$, and
$\esc{\cdot,\cdot}_\FC$ is partially continuous relative to the 
$\esc{\cdot,\cdot}_{0\FC}$-topology of $V_\FC$. As a consequence the next result follows.
\begin{corollary}
If $\F$ is a nondegenerate family of inner products in a $\K$-vector space $V$ and $\FC\supset\K$ a field extension, then the family
$\F_\FC$ is again nondegenerate.
\end{corollary}

\subsection{Construction by ultrafilters.}

Roughly speaking, in this subsection we start with a \so\ family $\F$ of inner products on a vector space $V$ over $\K$. Under suitable mild hypothesis, we find a field extension $\FC\supset\K$ and an expanded family $\F_\FC\cup\{\escd{\cdot,\cdot}_\FC\}$ which is nondegenerate in the sense of Definition \ref{fuerade}.

\medskip

\indent Recall that $(I,\le)$ is a {\it directed set}, if  it is a preordered set  satisfying $\forall i,j\in I, \exists \, k\in I\colon i,j\le k$.

\begin{remark}\label{onimuhc} \rm
 Let $\F=\{\esc{\cdot,\cdot}_i\}_{i\in I}$ be a family of inner products in a $\K$-vector space $V$. We may assume that $I$ is a directed set, in fact by Zermelo's Theorem (also known as the Well-Ordering Theorem), any set can be well-ordered. Since any well-ordered set is a directed set, we conclude that there is no loss of generality assuming that $I$ is a directed set for some preorder relation $\le$. 
\end{remark}

We recall that a {\it filter of subsets} ${\mathfrak F}$ of a given set $X$ is a collection ${\mathfrak F}\subset\text{\bf P}(X)$ (the {\it power set of} $X$) such that:
\begin{enumerate}[label=(\roman*)]
    \item $\emptyset\notin\mathfrak{F}, \; X\in \mathfrak{F}$.
    \item $S,S'\in\mathfrak{F}$ implies $S\cap S'\in\mathfrak{F}$.
    \item $S\in\mathfrak{F}$ and $S\subset S'$ implies $S'\in \mathfrak{F}$.
\end{enumerate}
 
 An {\it ultrafilter} on $X$, say $\mathfrak{U}$, is just a maximal filter of subsets of $X$ (equivalently for any $S\subset X$ one has $S\in\U$ or $X\setminus S\in\U$). For more information about filters see \cite[Chapter 3]{F}.
In these conditions we have the following definitions.

\begin{definition}\label{Leo} \rm
Let $\F=\{\esc{\cdot,\cdot}_i\}_{i\in I}$ be a family of inner products in a $\K$-vector space $V$. Note that $I$ is a directed set by Remark \ref{onimuhc}.
For each $i\in I$ define $[i,\to):=\{j\in I\colon i\le j\}$. Take the filter $\mathfrak{F}:=\{S\subset I\colon \exists \, \, i\in I,\, [i,\to)\subset S\}$ and an ultrafilter $\mathfrak{U}$ containing $\mathfrak{F}$. If we define $\K_i:=\K$ we will denote the elements of $\prod_i\K_i$ in the form $(x_i)_{i\in I}$ with $x_i \in \K$. The usual equivalence relation in  $\prod_i\K_i$ is given by $(x_i)_{i\in I}\equiv (y_i)_{i\in I}$ if and only if 
$\{i\in I\colon x_i=y_i\}\in\mathfrak{U}$. 
The equivalence class of $(x_i)_{i\in I}$ will be denoted $[(x_i)_{i\in I}]$ or $[(x_i)]$ for short. The quotient of $\prod_{i\in I}\K_i$ module the relation $\equiv$ will be denoted:
\begin{equation}\label{coneho}
\FC:=\prod_{i\in I}\K_i/\mathfrak{U}.
\end{equation}

\end{definition}

It is well known that $\FC\supset\K$ is a field extension for the operations:
$$[(x_i)]+[(y_i)]:=[(x_i+y_i)], \ [(x_i)][(y_i)]:=[(x_iy_i)].$$

Indeed the extension is realized via the canonical monomorphism $\K\to\FC$ such that $\lambda\mapsto\vec{\l}$ where $\vec{\l}$ denotes the equivalence class of the element $(x_i)_{i\in I}$ such that $x_i=\l$ for any $i$. For further details about the construction of the field $\FC$ see \cite[Chapter 3]{F}. 

\begin{remark} \label{estornudo} \rm
Assume that $I$ is finite and endow it with a well ordering. So there is no loss of generality assuming that 
$I=\{1,\ldots,n\}$ with its usual order as a subset of the natural numbers.
Let $\mathfrak{F}$ be a filter containing all the intervals 
$[i,\to)$ for $i=1,\ldots,n$. Then 
$$\mathfrak{F}=\{S\subset I\colon n\in S\}.$$
Moreover, this filter is an ultrafilter because any subset of $I$ contains $n$ or its complementary contains $n$. We will denote $\K^I=\prod_{i=1}^n\K_i$. So we take $\U=\mathfrak{F}$ and consider $\K^I/\U$ where now two $n$-tuples
$(x_i)_1^n$ and $(y_i)_1^n$ are related if and only if $x_n=y_n$.
    Therefore the canonical monomorphism $\K\to\FC:=\K^I/\U$ is an isomorphism since for any $[(x_i)_1^n]\in\FC$ we have $[(x_i)_1^n]=[(x_n,x_n,\ldots,x_n)]=\ray{x_n}$. In this case we do not get a proper extension field of $\K$ but $\K$ itself.
\end{remark}

\begin{definition} \rm
Let $\F=\{\esc{\cdot,\cdot}_i\}_{i\in I}$ be a family of inner products in the $\K$-vector space $V$. Observe that $I$ is a directed set by Remark \ref{onimuhc}. 
We will say that an element $x\in V$ is {\em pathological} if for any finite-dimensional subspace $S$ of $V$ the set $S_x:=\{i\in I\colon \esc{x,S}_i=0\}\in\mathfrak{U}$ where $\mathfrak{U}$ is the ultrafilter as in Definition \ref{Leo}. The set of all  pathological elements will be denoted $\path_{\F}(V)$. When $\path_{\F}(V)=0$ we will say that $\F$ is \em{nonpathological}.
\end{definition} 

It is easy to prove that the set of all pathological elements is a subspace of $V$.
Let $x,y\in\path_{\F}(V)$ then for any finite-dimensional subspace $S$ of $V$ consider $S_x=\{i\in I\colon \esc{x,S}_i=0\}$ and similarly $S_y$.
Then $S_x,S_y\in\mathfrak{U}$ so $S_x\cap S_y\in\mathfrak{U}$. We have $S_x\cap S_y\subset S_{x+y}$ whence $S_{x+y}\in\mathfrak{U}$ implying $x+y\in\path_{\F}(V)$.
Thus $\path_{\F}(V)+\path_{\F}(V)\subset\path_{\F}(V)$ and analogously we have $\K \path_{\F}(V)\subset\path_{\F}(V)$. Consequently choosing a complement to $\path_{\F}(V)$ in $V$ we have a decomposition $V=\path_{\F}(V)\oplus W$.
When $W=0$ all the elements are pathological and one can see that this is equivalent to the statement that for any $x,y\in V$ we have $\{i\in I\colon \esc{x,y}_i=0\}\in\U$.
When $W\ne 0$ we can consider $\F\vert_W$ and the question is: Is $\F\vert_W$ a nonpathological family?

\begin{proposition} Suppose $\F=\{\esc{\cdot,\cdot}_i\}_{i\in I}$ is a family of inner products in the $\K$-vector space $V$ where $V=\path_{\F}(V)\oplus W$. If $W\ne 0$ then $\F\vert_W$ is a nonpathological family.
\end{proposition}

\begin{proof} Let $a \in \path_{\F\vert_W}(W)$, this means that for any finite-dimensional subspace $R$ of $W$ the set $R_a=\{i\in I\colon \esc{a,R}_i=0\} \in \mathfrak{U}$. Consider a finite-dimensional subspace $S$ of $V$.
We want to prove that  $\{i\in I\colon \esc{a,S}_i=0\}\in \U$. Let  $\pi_W: V \to W$ be the projection onto $W$. Consider $R:=\pi_W(S)$ the projection of $S$ onto $W$. We know that $R_a\in\mathfrak{U}$.
Write $s= p + w \in S$ where $p \in \path_{\F}(V)$ and $w \in W$, then $\esc{a,s}_i=\esc{a,p+w}_i=\esc{a,p}_i+\esc{a,w}_i$. Since $p$ is pathological $(\K a)_p=\{i\in I\colon\esc{p,\K a}_i=0\}\in \U$. Hence $ R_a\cap (\K a)_p \in \U$. Now observe that if $i\in(\K a)_p$ we can say that 
$\esc{p,a}_i=0$ and in case $i\in R_a$ we can write 
$0=\esc{a,R}_i=\esc{a,\pi_W(S)}_i$ which implies $\esc{a,w}_i=0$. Consequently, if $i\in R_a\cap (\K a)_p \in \U$ we have 
$\esc{a,s}_i=\esc{a,p}_i+\esc{a,w}_i=0$. Hence for an fixed $s\in S$ we have $\{i\in I\colon \esc{a,s}_i=0\}\in\U$. To finish the proof we consider a basis $\{s_1,\ldots, s_k\}$ of $S$. We know that each set $\{i\in I\colon \esc{a,s_j}_i=0\}\in\U$ ($j=1,\ldots,k$), hence the intersection of these sets is contained in $\U$. We conclude that $\{i\in I\colon \esc{a,S}_i=0\}\in\U$. Thus $a\in\path_{\F}(V)\cap W=0$, so the family $\F\vert_W$ is nonpathological.
\end{proof}

\begin{definition} \rm
Consider a  family $\F=\{\esc{\cdot,\cdot}_i\}_{i\in I}$ of inner products in the $\K$-vector space $V$. We define a $\K$-bilinear symmetric map $\escd{\cdot,\cdot}\colon V\times V\to\FC$, being $\FC$ as in \eqref{coneho}, by
\begin{equation}\label{etnetop}
\escd{x,y}:=[(\esc{x,y}_i)_{i\in I}] \ \text{ for any }\ x, y \in V.
\end{equation}
\end{definition}

 \begin{proposition}Let $\F=\{\esc{\cdot,\cdot}_i\}_{i\in I}$ be a family  of inner products in the $\K$-vector space $V$ and the $\K$-bilinear symmetric map $\escd{\cdot,\cdot}\colon V\times V\to\FC$ defined in  \eqref{etnetop}. Then
 $ \path_{\F}(V)=\{x\in V\colon \escd{x,V}=0\}.$
\end{proposition}
\begin{proof}
If $x\in \path_{\F}(V)$ each $S_x\in\U$ (for arbitrary finite-dimensional subspace $S$). Take an arbitrary $v\in V$ then
$(\K v)_x\in\U$ hence $\{i\colon\esc{x,v}_i=0\}\in\U$ so 
$\escd{x,v}=0$ and since $v$ is arbitrary $\escd{x,V}=0$.
Conversely if $\escd{x,V}=0$ then take a finite-dimensional subspace $S$ of $V$. Consider a basis $\{s_1,\ldots,s_k\}$ of $S$.
Then $0=\escd{x,s_j}=[(\esc{x,s_j}_i)_{i\in I}]$ whence 
$\{i\in I\colon \esc{x,s_j}=0\}\in\U$ for $j=1,\ldots,k$. So
$\{i\in I\colon \esc{x,S}=0\}\in\U$ and $x\in \path_{\F}(V)$.
\end{proof}

\begin{definition} \rm
Let $\F=\{\esc{\cdot, \cdot}_i\}_{i\in I}$ be a family  of inner products in the $\K$-vector space $V$ and let $\escd{\cdot,\cdot}\colon V\times V\to\FC$ be as in \eqref{etnetop}. We will define on $V_\FC:=V \otimes \FC$ an inner product
$\escd{\cdot,\cdot}_\FC\colon V_\FC\times V_\FC\to \FC$ such that 
$$\escd{x\oo\lambda,y\oo\mu}_\FC:=\lambda\mu\escd{x,y} \ \text{ for any }\ x\oo\lambda,\  y\oo\mu \in V_\FC.$$ 
\end{definition}

\begin{theorem}\label{pajaro}Let $\F=\{\esc{\cdot,\cdot}_i\}_{i\in I}$ be a  family of inner products in the $\K$-vector space $V$. We have the following:
\begin{enumerate}[label=(\roman*)]

\item\label{papa} If $\F$ is nonpathological and \so, then $\escd{\cdot,\cdot}_\FC$ is nondegenerate.
\item \label{pipi} If $\escd{\cdot,\cdot}_\FC$ is nondegenerate then $\F$ is nonpathological.
\end{enumerate}
\end{theorem}

\begin{proof}
Let $B=\{v_j\}$ be an orthogonal basis of $V$ relative to any inner product of $\F$.
Then $\{v_j\oo 1\}$ is an orthogonal basis of $(V_\FC,\escd{\cdot,\cdot}_\FC)$ and
to prove that $\escd{\cdot,\cdot}_\FC$ is nondegenerate it suffices to prove that
$\escd{v_j\otimes 1,v_j\otimes 1}_\FC$ is nonzero for any $j$. But $\escd{v_j\otimes 1,v_j\otimes 1}_\FC=\escd{v_j,v_j}$. If we suppose that there exists $j$ such that $\escd{v_j,v_j}=0$ we can define 
$T:=\{i\in I\colon \esc{v_j,v_j}_i=0\}\in\U$. Next we prove that $v_j$ is a pathological element of $\F$: let $S$ be a finite-dimensional subspace of $V$, then there is a finite collection $v_{k_1},\ldots,v_{k_m}$ of elements of $B$ such that $S\subset\span(\{v_{k_q}\}_{q=1}^m)$. Now we see that $\{i\in I\colon \esc{v_j,S}_i=0\}\in\U$. Indeed, first $T\subset S_q:=\{i\in I\colon \esc{v_j,v_{k_q}}_i=0\}$ for each $q\in\{1,\ldots,m\}$. Thus $\cap_1^m S_q\in\U$ and $\cap_1^m S_q\subset \{i\in I\colon \esc{v_j,S}_i=0\}$. We conclude that $v_j\in\path_{\F}(V)=0$ which is a contradiction.

Next to prove \ref{pipi} we see that if $v\in\path_{\F}(V)$, then $v\oo 1\in\rad(\escd{\cdot,\cdot}_\FC)$. Indeed, we have
 $\escd{v\oo 1,x\oo 1}_\FC=\escd{v,x}=[(\esc{v,x}_i)]$ for 
$x\in V$. But $v$ is pathological so $\{i\colon \esc{v,x}_i=0\}\in\U$. Thus
$[(\esc{v,x}_i)]=0$ hence $v\oo 1\in\rad(\escd{\cdot,\cdot}_{\FC})$.
\end{proof}

\begin{remark}\label{aparruz}\rm 
 Note that if $\{v_j\}$ is a basis of $V$ diagonalizing $\F$, then $\{v_j\oo 1\}$ is
 a basis of $V_\FC$ diagonalizing  the inner product $\escd{\cdot,\cdot}_\FC$. In this way the family $\F_\FC\cup\{\escd{\cdot,\cdot}_\FC\}$ is \so. 
\end{remark}

In order to illustrate Theorem \ref{pajaro} we can consider the following example.

\begin{example}\rm
 Assume that $V$ is $\aleph_0$-dimensional and consider a family of inner products in $V$ given by $\F=\{\esc{\cdot,\cdot}_i\}_{i\in\N}$ where relative to a basis $B=\{v_j\}$ of $V$ we have $\esc{v_i,v_j}_k=0$ if $i\ne j$ for each $k$,
$\esc{v_j,v_j}_j=0$ and $\esc{v_j,v_j}_i=1$ for $i\ne j$. Thus the matrix of any $\esc{\cdot,\cdot}_i$ relative to $B$ is a diagonal matrix with $1$'s on the diagonal except for the $(i,i)$ entry which is $0$. Consequently each inner product of $\F$ is degenerate. Furthermore, $\F$ is nonpathological since 
if $x$ is a pathological element and we consider  an arbitrary $v_j\in B $, then $\{i\in \N \colon \esc{x,v_j}_i=0\} \in \U$, so $\{i\in \N \colon \esc{x,v_j}_i=0\}\ne \{j\}$ and  there exists $i\ne j$ such that $0=\esc{x,v_j}_i=\sum_{k}\l_k\esc{v_k,v_j}_i=\l_j\esc{v_j,v_j}_i=\l_j$ for arbitrary $j$, therefore $x=0$. 
Hence if we compute $\escd{v_j,v_i}$ we find 
$$\escd{v_j,v_i}=\begin{cases}1 & \hbox{if}\ i=j\\ 
0 & \hbox{if}\ i\ne j\end{cases}$$ so that $\escd{\cdot,\cdot}_\FC$ is nondegenerate.
\end{example}

Now we observe that when $\F$ is nonpathological and $I$ is finite, then $\F$ is nondegenerate.

\begin{remark}\rm
 Assume that $\F$ is nonpathological, $I$ is finite (concretely, without loss of generality that $I=\{1,\dots,n\}$) and
that the ultrafilter $\U$ consists of all the subsets of $I$ containing $n$. The $\K$-bilinear symmetric  map $\escd{\cdot,\cdot}$ is precisely $\esc{\cdot,\cdot}_n$ because
$\escd{x,y}=[(\esc{x,y}_i)_1^n]=[(\esc{x,y}_n,\ldots,\esc{x,y}_n)]$.
We can check that in this case $\esc{\cdot,\cdot}_n$ is nondegenerate: assume that for a nonzero $x\in V$ we have $\esc{x,V}_n=0$, since $x$ is not pathological there is some finite-dimensional subspace $S$ of $V$ such that $S_x\notin\U$.  
Then $n\notin S_x$ whence $n\in \complement S_x$ and so $\esc{x,S}_n\ne 0$ which is a contradiction.
\end{remark}

\begin{corollary} \label{playita} If $\F=\{\esc{\cdot,\cdot}_i\}_{i\in I}$ is a nonpathological  and \so\  family of inner products in the $\K$-vector space $V$, then the family $\F_\FC\cup\{\escd{\cdot,\cdot}_\FC\}$ is nondegenerate.
\end{corollary}

\begin{proof}
First observe that $\escd{\cdot,\cdot}_\FC$ is nondegenerate by Theorem \ref{pajaro} item \ref{papa}. In order prove the statement we need to check that any $\esc{\cdot,\cdot}_{i\FC}$ is partially continuous relative to the $\escd{\cdot,\cdot}_\FC$-topology. Since both  $\esc{\cdot,\cdot}_{i\FC}$ and $\escd{\cdot,\cdot}_{\FC}$ are \so\ (see Remark \ref{aparruz}), then applying Proposition \ref{atupoi} item \ref{atupoi2} we have that  {$\esc{ \cdot, \cdot  }_{i\FC}$} is partially continuous. Summarizing the family $\F_\FC\cup\{\escd{\cdot,\cdot}_{\FC}\}$ is nondegenerate. 
\end{proof}

\subsubsection{\textit{\textbf{The real-complex case.}}}
In this subsubsection we will study the case of the previous statement in which the field is $\R$ or $\C$. Their respective extended fields are: the hyperreal numbers denoted by $\hipR$ and hypercomplex numbers denoted by $\hipC$. Recall
that $\hipR=\R^I/\U$ (and similarly for $\hipC$) as in the previous subsection.
An element in $\hipR$ or in $\hipC$ can be termed a {\it hypernumber} so that we do not precise if it is an hyperreal or an hypercomplex number.
The absolute value in $\R$ and the complex modulus in $\C$ will be denoted by $\vert\cdot\vert$ (we will distinguish them only by the context). 
The extension of the absolute value function $\vert\cdot\vert$ to $\hipR$ will be denoted by abuse of notation
as $\vert\cdot\vert\colon\hipR\to {\hipR}$ (and similarly $\vert\cdot\vert\colon\hipC\to{\hipC}$). 
This extension is given by $\vert [(x_i)]\vert:=[(\abs{x_i})]$ for any $[(x_i)]$.
Also, given two hyperreals $x,y\in{\hipR}$ with 
$x=[(x_i)]$ and $y=[(y_i)]$,
it is said that $x<y$ if $\{i\in I\colon x_i<y_i\}\in\U$.
Recall some basic facts on infinite and infinitesimal hypernumbers: let $\K=\R$ or $\C$, an element $x\in\hipK$ is said to be {\em finite} if there is some real $M\in\R$ such that $\vert x\vert<M$. An element $x\in\hipK$ is said to be an {\em infinitesimal} is for any real $\epsilon>0$ we have $\abs{x}<\epsilon$. A known result is that any finite hypernumber $x\in\hipK$ can be written of the form $x=\st(x)+w$ where $w$ is an infinitesimal and $\st(x)\in\K$. The element $\st(x)$ is unique and is called the
{\em standard part} of $x$. Note also that if $x,y$ are finite hypernumbers then $x+y$ is a finite hypernumber and $\st(x+y)=\st(x)+\st(y)$. Also if $\l\in\K$ and $x$ is a finite element of $\hipK$ then $\l x$ is finite and $\st(\l x)=\l\st(x)$.
Finally we recall the extended notation $x\approx y$ for hypernumbers whose difference is an infinitesimal (consequently $\st(x)=\st(y)$).
 \begin{definition}\rm{Let $V$ be a $\K$-vector space for $\K= \R, \C$ and consider a family of inner products $\F=\{\esc{\cdot,\cdot}_i\}_{i \in I}$ of $V$. We recall that $I$ is a directed set by Remark \ref{onimuhc}. We say that $\F$ is \emph{bounded} if for $x,y \in V$ there exists $M \in \R$ such that $\{ i \in I \colon \vert \esc{x,y}_i \vert < M \} \in \U$} being $\U$ an ultrafilter as in Definition \ref{Leo}.
\end{definition}
In case $I$ is finite, it is trivial that any family $\F$ indexed by $I$, is bounded. So the definition is meaningful only if $I$ is an infinite set. 
In general, $\F$ is bounded if and only if for any $x,y\in V$, the hypernumber 
$[(\esc{x,y}_i)]$ is finite. Defining $\escd{\cdot,\cdot}\colon V\times V\to\hipR$
as in formula \eqref{etnetop} we can say that $\F$ is bounded if and only if for any $x,y\in V$, we have that $\escd{x,y}$ is finite.
\begin{definition}\rm{Given the family $\F=\{\esc{\cdot,\cdot}_i\}_{i \in I}$, 
consider an ultrafilter $\U$ as Definition \ref{Leo}, containing all the intervals $[i,\to)$ with $i\in I$. An element $x \in V$ is said to be \emph{negligible} if for every $y \in V$ and for every real $\epsilon > 0$, we have $\{i \in I \colon \vert \esc{x,y}_i \vert < \epsilon\} \in \U$. We will denote the set of all negligible elements by $\neg_\F(V)$. When $\neg_\F(V)=0$ we will say that $\F$ is \emph{robust}.}
\end{definition}

An element $x$ is negligible if and only if for any $y\in V$ the hypernumber
$\escd{x,y}$ (as in formula \eqref{etnetop} again) is an infinitesimal. Observe that $\neg_{\tiny\F}(V)$ is a subspace of $V$ and that if $x \in \path_{\F}(V)$ then $x$ is negligible. So we can write $V=\neg_\F(V) \oplus W$ for certain subspace $W \subset V$. Furthermore $\F\vert_ W$ is robust.

\begin{definition}\rm 
If $\F$ is a bounded family we can define an inner product $\esct{\cdot,\cdot}\colon V\times V\to\R$ such that $\esct{x,y}:=\st(\escd{x,y})$ for any $x,y\in V$.
\end{definition}

As a final result we have the following one.

\begin{theorem}\label{wwe} Let $V$ be a $\K$-vector space for $\K= \R, \C$ and consider a bounded family of inner products $\F=\{\esc{\cdot,\cdot}_{i}\}_{i \in I}$ of $V$. Then:
\begin{enumerate}[label=(\roman*)]
\item\label{wweuno} We have $\rad\esct{\cdot,\cdot}=\neg\nolimits_\F(V)$. 
\item\label{wwedos}  If $\F$ is robust, the inner product $\esct{\cdot,\cdot}$ is nondegenerate.
\item\label{wwetres} $\F$ is \so\ if and only if the enlarged family $\F\cup\{\esct{\cdot,\cdot}\}$ is \so. Furthermore, if $\F$ is robust and \so, the enlarged family $\F\cup\{\esct{\cdot,\cdot}\}$ is nondegenerate.
\item\label{wwecuat} The simultaneous orthogonalizations of $V$ relative to $\F$ and to $\F\cup\{\esct{\cdot,\cdot}\}$ agree. 
\end{enumerate}
\end{theorem}
\begin{proof}
To prove the equality in \ref{wweuno} consider $x\in\rad\esct{\cdot,\cdot}$. This
happens exactly when $\escd{x,y}\approx 0$ for any $y$. Equivalently, for any real $\epsilon>0$ the set  $\{i\in I\colon\vert\esc{x,y}_i\vert<\epsilon\}$ is in $\U$. So $x$ is negligible and reciprocally. Now \ref{wwedos} is trivial from the previous item.
To prove \ref{wwetres}, take a basis $\{v_j\}$ of $V$ orthogonal relative to $\F$.
Then if $j\ne k$ we have $\esct{v_j,v_k}=\st(\escd{v_j,v_k})$. But $\escd{v_j,v_k}=[(\esc{v_j,v_k}_i)]=0$ hence $\esct{v_j,v_k}=0$. Thus the same basis
is also orthogonal for $\esct{\cdot,\cdot}$. Assume besides that $\F$ is robust. The inner products $\esc{\cdot,\cdot}_i$ are partially continuous relative to the $\esct{\cdot,\cdot}$-topology because of item \ref{wwedos} of this theorem and Proposition \ref{atupoi} item \ref{atupoi2}. Finally $\F\cup\{\esct{\cdot,\cdot}\}$ is nondegenerate. For the assertion in \ref{wwecuat} take into account that any orthogonal
basis $B$ relative to $\F$ is automatically an orthogonal basis of $\F\cup\{\esct{\cdot,\cdot}\}$ and reciprocally.
\end{proof}

\end{document}